\documentclass[reqno,11pt]{amsart}
\usepackage{amsfonts}
\usepackage{graphicx}
\usepackage{amsfonts,amsmath, amssymb}
\usepackage{bbm,mathrsfs,mathscinet}
\usepackage[colorlinks,linkcolor=blue,anchorcolor=red,citecolor=blue]{hyperref}
\usepackage{marginnote}
\usepackage{color}
\usepackage{placeins}
\usepackage{verbatim}
\usepackage{tikz}
\usetikzlibrary{decorations,arrows}
\allowdisplaybreaks

\usepackage[margin=1.1in]{geometry}

\usepackage{hyperref}
\hypersetup{hidelinks}

\numberwithin{equation}{section}

\newcommand{\R}{\mathbb{R}}

\newcommand{\cal}{\mathcal}

\makeatletter
\newcommand{\rmnum}[1]{\romannumeral #1}
\newcommand{\Rmnum}[1]{\expandafter\@slowromancap\romannumeral #1@}
\makeatother

\newtheorem{theorem}{Theorem}[section]

\newtheorem{lemma}[theorem]{Lemma}
\newtheorem{proposition}[theorem]{Proposition}

\theoremstyle{remark}
\newtheorem{remark}{Remark}[section]

\theoremstyle{definition}

\def\x{\xi}

\def\pa{\partial}

\begin{document}
	\title[Axisymmetric subsonic Euler flows]
	{Global three-dimensional subsonic Euler flows past an axisymmetric obstacle with large vorticity}
	
	\author[D.-H. Wang]{Dehua Wang}
	\address{Department of Mathematics, University of Pittsburgh, Pittsburgh, PA 15260, USA.}
	\email{dhwang@pitt.edu}
	
	\author[T.-Y. Wang]{Tian-Yi Wang}
	\address{School of Mathematics and Statistics, 
		Wuhan University of Technology, Wuhan, Hubei 430070, China.}
	\email{tianyiwang@whut.edu.cn}
	
	\author[W.-Q. Wang]{Weiqiang Wang}
	\address{Department of Mathematics, University of Pittsburgh, Pittsburgh, PA 15260, USA.}
	\email{wew179@pitt.edu}
	
	\begin{abstract}
		In this paper, we prove the existence and uniqueness of subsonic solutions to the steady Euler flows past a smooth, axisymmetric obstacle. Specifically, for a broad class of prescribed positive axial velocities in the upstream, the subsonic Euler flow exists provided that the upstream density exceeds a critical threshold. The non-degeneracy of the axial velocity is rigorously established by combining the strong maximum principle with a refined continuity argument. The asymptotic behavior of the flow is obtained from uniform integral estimates for the difference between the flow and the upstream state. In addition, this result accommodates flows with large vorticity under a structural condition, thereby differing from previous results in the two-dimensional case.
	\end{abstract}

	\subjclass[2020]{35Q31, 35M30, 76N10, 76G25, 35B40}
	\keywords{Compressible Euler equations; Subsonic flow; Axisymmetric obstacle; Large vorticity}
	\date{\today}
	\maketitle
	
	\setcounter{tocdepth}{2}
	\tableofcontents

	\section{Introduction}
	
	In this paper, we consider the existence and uniqueness of the three-dimensional subsonic flow past an axisymmetric obstacle,  governed by the following steady compressible isentropic Euler equations:
	\begin{equation}\label{1.1}
		\left\{
		\begin{aligned}
			&\operatorname{div}(\rho\mathbf{U})=0,\\
			&\rho (\mathbf{U}\cdot \nabla )\mathbf{U}+\nabla P=0,
		\end{aligned}
		\right.\quad {\bf x}=(x,y,z)\in \Omega_{0},
	\end{equation}
	subject to the following boundary condition:
	\begin{equation}\label{1.2}
		\mathbf{U}\cdot \vec{n}=0,\,\,\text{on }\partial\Omega_0,
	\end{equation}
	which is equivalent to 
	\begin{equation}\label{1.2+}
		\rho \mathbf{U}\cdot \vec{n}=0,\,\,\text{on }\partial\Omega_0
	\end{equation}
	for flows without vacuum,
	and far-field conditions in the upstream
	\begin{equation}\label{1.2-1}
		\lim\limits_{x\to -\infty}\mathbf{U}	(x,y,z)=(u_{\infty}(r),0,0),\quad \lim\limits_{x\to -\infty}\rho(x,y,z)=\rho_{\infty}.
	\end{equation}
	Here, $\rho$, $\mathbf{U}$ and $P$ are the density, velocity, and pressure, respectively; the pressure-density relation satisfies the $\gamma$-law:
	$$
	P(\rho)=\frac{1}{\gamma}\rho^{\gamma},
	$$
	where $\gamma>1$ is the adiabatic index. 
	The domain 
	$$
	\Omega_{0}=\{(x,y,z)\,|\,x\in \R,\,\,r=\sqrt{y^2+z^2}> f(x)\}
	$$
	is the complement of a finite axisymmetric obstacle in $\R^3$, see Figure \ref{fig1}, and 
	$$
	\vec{n}=\frac{1}{\sqrt{1+|f'(x)|^2}}\Big(f'(x),-\frac{y}{r},-\frac{z}{r}\Big)
	$$ 
	is the unit outward normal vector of $\Omega_0$. In this paper, the function $f(x)$ is assumed to satisfy
	\begin{align}\label{1.2-2}
		f(x)\geq 0,\quad  f(x)\equiv 0\quad \text{for $x\leq 0$ and $x\geq 1$},\quad f\in C^{2,\nu}(\R)\quad \text{for some $0<\nu<1$}.
	\end{align}
	
	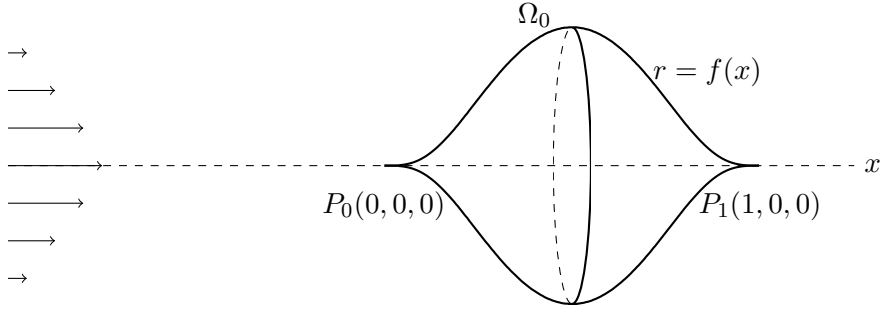
\begin{figure}[htbp]
		\centering
		\begin{tikzpicture}[scale=5]
			
			\draw[dashed] (-1,0) -- (1.25,0) node[right] {$x$};
			
			\draw[thick,black,domain=0.001:0.999,samples=200]
			plot (\x,{8*exp(-1/((\x+0.07)*(1.07-\x)))});
			\draw[thick,black,domain=0.001:0.999,samples=200]
			plot (\x,{-8*exp(-1/((\x+0.07)*(1.07-\x)))});
			
			\foreach \x in {0.5}
			{
				\pgfmathsetmacro{\rad}{-8*exp(-1/((\x+0.07)*(1.07-\x)))}
				\pgfmathsetmacro{\xr}{0.05}
				
				\begin{scope}
					\clip (\x-\xr,-\rad) rectangle (\x, \rad);
					\draw[dashed]
					(\x,0) ellipse[x radius=\xr, y radius=\rad];
				\end{scope}
				
				\begin{scope}
					\clip (\x, -\rad) rectangle (\x+\xr, \rad);
					\draw[thick]
					(\x,0) ellipse[x radius=\xr, y radius=\rad];
				\end{scope}
			}
			
			\draw[->] (-1,0)--(-0.75,0);
			\draw[->] (-1,0.1)--(-0.8,0.1);
			\draw[->] (-1,0.2)--(-0.875,0.2);
			\draw[->] (-1,0.3)--(-0.95,0.3);
			\draw[->] (-1,-0.1)--(-0.8,-0.1);
			\draw[->] (-1,-0.2)--(-0.875,-0.2);
			\draw[->] (-1,-0.3)--(-0.95,-0.3);
			
			\node at (0.4,0.4) {$\Omega_{0}$};
			\node at (0.86,0.25) {$r=f(x)$};
			\node[below] at (0,-0.04) {$P_{0}(0,0,0)$};
			\node[below] at (1,-0.04) {$P_{1}(1,0,0)$};
		\end{tikzpicture}
		\caption{Subsonic Euler flows past an axisymmetric obstacle}
		\label{fig1}
	\end{figure}
	
	In light of the axial symmetry of $\Omega_0$, it is natural to look for axially symmetric solutions to \eqref{1.1} (cf. \cite{DD-2011,DD-2012,DD-2016}):
	$$
	\rho=\rho(x,r),\quad \mathbf{U}=\Big(u(x,r),v(x,r)\frac{y}{r},v(x,r)\frac{z}{r}\Big),
	$$
	where $u$ is the axial velocity and $v$ is the radial velocity. Then \eqref{1.1} and \eqref{1.2}--\eqref{1.2-1} can be rewritten in cylindrical coordinates $(x,r)$ as follows:
	\begin{align}\label{1.3}
		\left\{
		\begin{aligned}
			&\partial_{x}(r\rho u)+\partial_{r}(r\rho v)=0,\\
			&(r\rho u^2)_{x}+(r\rho u v)_{r}+rP_{x}=0,\\
			&(r\rho uv)_{x}+(r\rho v^2)_{r}+rP_{r}=0,
		\end{aligned}
		\right.\quad (x,r)\in \Omega=\{(x,r)\,|\,x\in \R,\,\, r>f(x)\}
	\end{align}
	supplemented with the boundary condition:
	\begin{align}\label{1.3-1}
		f'(x)u-v=0\quad \text{on $\Gamma:=\partial\Omega=\{(x,r)\,|\,r=f(x)\}$}, 
	\end{align}
	and the far-field conditions in the upstream
	\begin{align}\label{1.3-2}
		\lim\limits_{x\to -\infty}(u(x,r),v(x,r))=(u_{\infty}(r),0),\quad \lim\limits_{x\to -\infty}\rho(x,r)=\rho_{\infty}.	
	\end{align}
	We would like to point out that if the radial velocity $v$ in the upstream is zero, then it follows directly from $\eqref{1.3}_{3}$ that $P_{r}=0$ in the upstream, which implies that $\rho_{\infty}$ must be a positive constant.
	
	We denote by $c(\rho)=\sqrt{P'(\rho)}=\rho^{\frac{\gamma-1}{2}}$ the sound speed.
	Then the Mach number of the flow is defined by
	$$
	M_{a}=\frac{\sqrt{u^2+v^2}}{c(\rho)}=\frac{q}{\rho^{\frac{\gamma-1}{2}}},
	$$
	where $q=\sqrt{u^2+v^2}$ is the speed of the flow. In this paper, we assume the upstream flow is uniformly subsonic, \textit{i.e.}, the Mach number in the upstream is strictly less than 1 for any $r\geq 0$:
	\begin{align}\label{Ma}
		M_{a}^{\infty}=\frac{\sup_{r\geq 0}u_{\infty}(r)}{\rho_{\infty}^{\frac{\gamma-1}{2}}}<1\; \Longleftrightarrow \; \rho_{\infty}> \rho_{\infty}^{*}=:\Big(\sup_{r\geq 0}u_{\infty}(r)\Big)^{\frac{2}{\gamma-1}}.
	\end{align}
	
	Subsonic Euler flows past an obstacle are of fundamental importance in aerodynamics \cite{B1,CF}. Although the mathematical study of this problem has a long history, most works concentrate on isentropic and irrotational flows. The first existence result for two-dimensional subsonic irrotational flows past a smooth body, assuming a small upstream Mach number, was obtained by Shiffman \cite{S1952}. Subsequently, Bers \cite{B2} removed the smallness restriction on the upstream Mach number, establishing existence for the full range of subsonic speeds up to a critical value via quasiconformal mappings. The general uniqueness and asymptotic behavior of such flows were later investigated in \cite{FG-1}; see also \cite{GS-1954} for additional properties. For higher-dimensional cases, Finn and Gilbarg \cite{FG-2} first derived asymptotic estimates and uniqueness criteria for subsonic flows, under certain restricted Mach number ranges. The existence of such flows was established by Dong and Ou \cite{DO-1993}, via Hilbert space variational methods, obtaining uniform subsonic irrotational flows for the full range of Mach numbers below a critical value. Employing the compensated compactness framework, the existence of weak solutions for subsonic–sonic irrotational flows was established in \cite{CDSW-2007,HWW-2011}.
	For transonic potential flows, the seminal work of Morawetz \cite{M1956,M1957,M1958} demonstrates that, in general, smooth transonic flows past a profile are unstable under small perturbations of the profile. Further developments concerning the Morawetz problem can be found in \cite{CDSW-2008,CGS}. 
	
	Another classical problem concerns subsonic flows in infinitely long nozzles, first posed by Bers in his survey \cite{B1} and subsequently extensively investigated. Xie and Xin \cite{XX-2007,XX-2010-1,XX-2010-2} established the existence of two-dimensional potential flows, isentropic Euler flows with small vorticity, and three-dimensional axisymmetric potential flows in such nozzles.  The case of two-dimensional Euler flows with large vorticity was later settled in \cite{DXX-2014}; see also \cite{DX-2014} for Euler flows with stagnation points in nonsmooth nozzles. For three-dimensional axisymmetric nozzles, Du and Duan \cite{DD-2011,DD-2012,DD-2016} studied both irrotational flows and rotational flows with small or large vorticity. For general multidimensional nozzles without symmetry assumptions, the existence and uniqueness of global subsonic potential flows were obtained in \cite{DXY-2011}. These results have been further extended to non-isentropic Euler flows with or without swirl  \cite{CDX-2012,DWX-2018,DL-2013,DL-2015,DW-2018,Z-2022}, as well as to subsonic flows in finitely long or periodic nozzles \cite{CX-2012,CX-2014,DWX-2014}.
	
	Regarding rotational Euler flows in exterior domains, the two-dimensional theory has been developed for half-plane configurations \cite{C-2011}, and for flows past a wall or a symmetric body \cite{CDXX-2016}. The present paper considers the well-posedness of subsonic Euler flows with large vorticity past an axisymmetric obstacle. The main results are stated as follows.

	\begin{theorem}\label{thm1}
		Suppose that the upstream axial velocity $u_{\infty}(r)$ in \eqref{1.3-2} satisfies 
		\begin{align}\label{1.3-3}
			u_{\infty}\in C^{2,\alpha}([0,\infty)),\quad u_{\infty}(r)>0,\quad u_{\infty}'(0)=0,\quad\lim\limits_{r\to \infty}u_{\infty}(r)=\bar{u}>0,
		\end{align}
		and the following structural condition:
		\begin{align}\label{1.3-4}
			u_{\infty}''(r)r\geq u_{\infty}'(r)
		\end{align}
		for $r\ge 0$. There exists a critical value $\rho_{cr}>\rho_{\infty}^{*}>0$, such that if the upstream density $\rho_{\infty}$ in \eqref{1.3-2} is larger than $\rho_{cr}$, 
		the Euler system \eqref{1.3} with the boundary conditions \eqref{1.3-1}--\eqref{1.3-2} admits a subsonic solution $(\rho,u,v)\in (C^{1,\alpha}(\Omega)\cap C^{\alpha}(\bar{\Omega}))^3$ for some $\alpha\in (0,\nu)$.
		Moreover,
		\begin{itemize}
			\item [(1)] the flow is uniformly subsonic:
			\begin{equation}\label{1.3-5}
				\sup_{(x,r)\in \overline{\Omega}}(u^2+v^2-c^2(\rho))<0.
			\end{equation}
			And the axial velocity is positive:
			\begin{align}\label{1.3-6}
				u>0\quad \text{in }\Omega \cup \{(x,r)\,|\,0<x<1, r=f(x)\}.
			\end{align}
			\item[(2)] the flow satisfies
			\begin{align}\label{1.3-7}
				\|r^{\frac{1}{2}}(\rho u-\rho_{\infty}u_{\infty},\rho v)\|_{L^2(\Omega)}\leq C
			\end{align}
			for some constant $C>0$ and has the following asymptotic behaviors in far fields:
			\begin{align}\label{1.3-8}
				\quad\,\,(\rho,u,v)(x,r)\to (\rho_{\infty},u_{\infty}(r),0),\,\, (\nabla\rho,\nabla v)(x,r)\to (0,0,0,0),\,\, \nabla u(x,r)\to (0,u_{\infty}'(r)),
			\end{align}
			as $|x|\to \infty$ uniformly for $r$ in any compact set $K\subseteq (0,\infty)$, and
			\begin{align}\label{1.3-9}
				u(x,r)\to \bar{u},\quad v(x,r)\to 0,\quad \rho(x,r)\to \rho_{\infty},
			\end{align}
			as $r\to \infty$ uniformly for $x$ in any compact set $K'\subseteq \mathbb{R}$.	
			\item[(3)] the subsonic flow satisfying the Euler system \eqref{1.3}, boundary conditions \eqref{1.3-1}--\eqref{1.3-2}, \eqref{1.3-6}--\eqref{1.3-7} and the asymptotic behavior \eqref{1.3-8}--\eqref{1.3-9} is unique.
			\item[(4)] $\rho_{cr}$ is the critical upstream density for the existence of subsonic flow past an axisymmetric obstacle in the following sense: either
			\begin{align}\label{1.3-10}
				\sup_{(x,r)\in \overline{\Omega}}\frac{\sqrt{u^2+v^2}}{c(\rho)}\to 1\text{ as }\rho_{\infty}\downarrow \rho_{cr},
			\end{align}
			or there is no constant $\sigma>0$, such that for all $\rho_{\infty}\in (\rho_{cr}-\sigma,\rho_{cr})$ there are Euler flows satisfying \eqref{1.3}, subsonic condition \eqref{1.3-5} and asymptotic behavior \eqref{1.3-7}--\eqref{1.3-9}, and
			$$
			\sup_{\rho_{\infty}\in (\rho_{cr}-\sigma,\rho_{cr})}\sup_{(x,r)\in\overline{\Omega}}\frac{\sqrt{u^2+v^2}}{c(\rho)}<1.
			$$
		\end{itemize}
	\end{theorem}
	
	As $\rho_{\infty}\downarrow \rho_{cr}$ in Theorem \ref{thm1}, using the compensated compactness framework developed in \cite{CHW-2016}, we have the following theorem.
	
	\begin{theorem}\label{thm2}
		Assume that $\{(\rho_{n},\mathbf{U}_{n})\}$ is a sequence of subsonic Euler flows established in Theorem \ref{thm1} with $\rho_{\infty}$ replaced by $\rho_{\infty}^{n}$. Suppose that $\rho_{\infty}^{n}\downarrow \rho_{cr}$ as $n\to \infty$, then there exists a limiting flow $(\rho,\mathbf{U})$ satisfying
		$$
		(\rho_{n}, \mathbf{U}_{n})\to (\rho, \mathbf{U})\quad \text{almost everywhere in $\Omega$ as $n\to \infty$}.
		$$
		Furthermore, $(\rho, \mathbf{U})$ satisfies the Euler system \eqref{1.1} in the sense of distribution and the boundary condition \eqref{1.2+} in the sense of normal trace.
	\end{theorem}
	
	\begin{remark}\label{rem3}
		Our results are also valid for global subsonic Euler flows with large vorticity past a bump: $\Omega_{1}=\{(x,y,z)\,|\,x\in \R,\,z\geq 0,\,0\leq \sqrt{y^2+z^2}\leq f(x)\}$ with $f(x)$ given in \eqref{1.2-2}.
	\end{remark}
	
	\begin{remark}\label{rem4}
		In view of \eqref{Ma}, given the upstream axial velocity, the density is greater than a critical value $\rho_{cr}$ if and only if the upstream Mach number $M_{a}^{\infty}$ is less than a critical value $M_{a,cr}^{\infty}$.
	\end{remark}
	
	\begin{remark}\label{rem1}
		The structural condition \eqref{1.3-4} implies that $$\frac{{\rm d}}{{\rm d}r}\big(\frac{u_{\infty}'(r)}{r}\big)\geq 0$$ for $r>0$. This together with \eqref{1.3-3}, which indicates $\lim\limits_{r\to\infty}\frac{u_{\infty}'(r)}{r}=0$, leads to $u_{\infty}'(r)\leq 0$, \textit{i.e.}, the upstream axial velocity is decreasing with respect to $r$. This is crucial for applying the strong maximum principle to show the non-degeneracy of axial velocity; see Lemmas \ref{lem2.3}--\ref{lem2.2} and Subsection 5.3 for details. 
	\end{remark}
	
	\begin{remark}\label{rem5}
		The upstream axial velocity $u_{\infty}(r)$ satisfying \eqref{1.3-3} and \eqref{1.3-4} can have a large vorticity $u_{\infty}'(r)$. For instance, 
		$$
		u_{\infty}(r)=\bar{u}+K(re^{-r}+e^{-r}\big),\quad u_{\infty}'(r)=-Kre^{-r},
		$$
		where $K$ is a large positive constant. Therefore, the Euler flows established in Theorem \ref{thm1} can substantially differ from the potential case \cite{DO-1993}.
	\end{remark}
	
	\begin{remark}\label{rem2}
		In the two-dimensional setting, the study of subsonic Euler flows with large vorticity past a wall \cite[Theorem~1.1]{CDXX-2016} relies crucially on the convexity condition $u_{\infty}''\geq 0$   for the upstream axial velocity. And, for flows past a symmetric body \cite[Corollary~3]{CDXX-2016}, the convexity condition is replaced by the small perturbation condition around a potential flow.
		For three-dimensional axisymmetric flow, our result employs the structural condition \eqref{1.3-4}, which allows the upstream flow to carry large vorticity without being restricted to a small neighborhood of a constant state. Such a structural condition has previously appeared in the study of large vorticity flows in the axisymmetric nozzles \cite{DD-2016}, and also in the non-isentropic case with non-zero swirl \cite{DWX-2018}.
	\end{remark}
	
	We now present an overview of the main components of our strategy and proofs.
	
	(\rmnum{1}) Using the stream function formulation developed in \cite{XX-2010-1}, the steady Euler system is transformed into a single second-order quasilinear equation with source term depending on the stream function, its gradient, and $r$ in the case of subsonic flow, where
	the hyperbolicity of the particle path is already involved. Since $\Omega$ is unbounded in $r$ direction, the stream function may be unbounded as $r\to \infty$, so we first construct the approximated problem in an infinitely long nozzle $\Omega_{L}=\Omega\cap \{(x,r)\,|\,x\in \R,\,\,0\leq r\leq L\}$. As in \cite{DD-2011,DD-2016}, to show the existence of subsonic flow in the nozzle $\Omega_{L}$, not only the subsonic truncation, but also the truncation near the axis of symmetry $r=0$ should be considered to avoid the possible degeneracy at sonic state and singularity at the axis of symmetry. The subsonic truncation can be similarly removed as in \cite{DD-2011,DD-2016} through analyzing the order with respect to $\rho_{\infty}$ of the Mach number.
	To handle the possible singularity, we construct a barrier function near the axis of symmetry. Since the upstream velocity is decreasing, \textit{i.e.}, $u_{\infty}'(r)\leq 0$, and the source term is negative, delicate analysis is needed in constructing the barrier function. Specifically, an auxiliary cubic term in $r$ should be added to the construction of the barrier function; see step 5 in the proof of Lemma \ref{lem2.1} for details.
	
	(\rmnum{2}) To take the limit $L\to \infty$ of the approximate solution in the nozzle $\Omega_{L}$, we need to establish uniform estimates with respect to $L$. As in \cite{CDXX-2016}, we consider the difference between the stream function and the upstream stream function. However, unlike the 2D case \cite{CDXX-2016}, we must take the weight $r$ into account when establishing $L^{\infty}$-estimates, H\"{o}lder gradient estimates, and $L^2$-estimates for this difference. In particular, for the H\"{o}lder gradient estimates, the source terms in the linear elliptic equation for this difference do not have the desired decay rate with respect to $r$ near the axis. Therefore, when $r$ is near the axis, we investigate the stream function itself rather than the difference and apply Moser's iteration to obtain the gradient estimate; see Lemma \ref{lem2} for details. Similar considerations are also employed in the proof of the uniqueness of subsonic Euler flows; see Subsection 5.3 for details. 
	
	(\rmnum{3}) To show that the stream function formulation is equivalent to the original Euler equation, one of the most important points is to show that the axial velocity is positive away from the axis. The proof of such non-degeneracy is quite different from the 2D case \cite{CDXX-2016}. In two dimensions, the axial velocity can be recovered as $u=\frac{\partial_{x_{2}}\psi}{\rho}$ with no possible singularity, where $\psi$ is the stream function and $\rho$ the density. Thus, the non-degeneracy of the axial velocity can be easily proved by applying the partial derivative $\pa_{x_{2}}$ to the stream function equation and using energy estimates, the strong maximum principle, the structural condition, and the far-field asymptotic behavior. However, in three dimensions under the axisymmetric setting, the axial velocity $u=\frac{\pa_{r}\psi}{r\rho}$, so the non-degeneracy of the axial velocity is not immediately evident from the equation for $\psi_{r}$, since $\psi_{r}=r\rho u$ itself is degenerate at the axis. In fact, the coefficients in the equation for $\psi$ depend on $r$, and the associated term becomes uncontrolled after applying the partial derivative $\partial_{r}$ in the energy estimate. Motivated by \cite{DWX-2018}, to overcome this issue, we first consider the irrotational case with coefficient $\frac{1}{r}$ truncated as $\frac{1}{r+k}$, so that the associated $\pa_{r}\psi$ can be proved non-degenerate everywhere, including the axis, through standard analysis of irrotational flows as in \cite{XX-2010-1}. Then, we consider the upstream flow as a small perturbation of the irrotational case, so that the associated $\partial_{r}\psi$ is non-degenerate everywhere provided the perturbation is sufficiently small. Finally, by applying an elegant continuity argument together with a proof by contraction and the intrinsic strong maximum principle without requiring the sign of the zeroth term, we succeed in extending the case of a small perturbation around the potential flow to the case of any large perturbation. Here, it is also worth pointing out that the decreasing property of the upstream axial velocity is essential for the validity of the strong maximum principle without requiring the sign of the zeroth term for the equation of $\partial_{r}\psi$; see Lemma \ref{lem2.3} and Subsection 5.2 for details.
	
	The rest of this paper is organized as follows. We first adapt the stream function formulation to reduce the steady Euler equations to a single second-order quasilinear elliptic equation with a source term depending on the stream function, its gradient, and $r$ in Section 2. The existence of a subsonic solution to the approximated problem in the nozzle $\Omega_{L}$ is established in Section 3. Subsequently, uniform estimates on the difference between the stream function and the upstream stream function with respect to $L$ are shown in Section 4. The existence of a subsonic solution in the exterior domain $\Omega$, along with fine properties, including the asymptotic behaviors in the far fields, the positivity of axial velocity away from the axis, the uniqueness of the solution, and the existence of critical density in the upstream, will be given in Section 5. In Section 6, we will use the compensated compactness framework developed in \cite{CHW-2016} to prove Theorem \ref{thm2}.
	
	{\it Notation}: We denote by $C^{k,\alpha}(\Omega)$  the standard H\"{o}lder space on $\Omega$, and $L^{p}(\Omega)$ with $1\leq p\leq \infty$ the standard $L^p$ space on $\Omega$. The constant $C$ denotes a constant independent of $L$
	and the elliptic coefficients. The constant $\mathscr{C}$ denotes a constant which is independent of $L$ but depends on the elliptic coefficients, and $\mathcal{C}$ is a constant depending on both $L$ and
	the elliptic coefficients. The value of these constants may vary from line to line, but they keep the same property. $A\sim O(\rho_{\infty})$ means $\frac{1}{C}\rho_{\infty}\leq A\leq C\rho_{\infty}$ for some constant $C$ independent of $L$ and the elliptic coefficients.

	\section{Stream Function Formulation}
	
	In this section, we will adapt the stream function formulation to the steady Euler system \eqref{1.3}. It follows from $\eqref{1.3}_{2}$--$\eqref{1.3}_{3}$ that
	\begin{align}\label{1.4-0}
		\rho u\partial_{x}B+\rho v\partial_{r}B=0,
	\end{align}
	where
	\begin{align}\label{1.4}
		B=\frac{1}{2}q^2+h(\rho)=:\frac{1}{2}(u^2+v^2)+\frac{1}{\gamma-1}\rho^{\gamma-1}
	\end{align}
	is the Bernoulli function, and $h(\rho)=\frac{1}{\gamma-1}\rho^{\gamma-1}$ is the enthalpy. Furthermore, denoting the vorticity $\omega=\partial_{x}v-\partial_{r}u$,  it follows from direct calculations and $\eqref{1.3}$ that
	\begin{align}\label{1.4-1}
		u\partial_{x}\Big(\frac{\omega}{\rho r}\Big)+v\partial_{r}\Big(\frac{w}{\rho r}\Big)=0.
	\end{align}
	Using similar arguments as in \cite[Proposition 2.1]{CDXX-2016}, we have the following proposition. 
	\begin{proposition}\label{Equivlence}
		Suppose that $(\rho,u,v)\in (C^{1,\alpha}(\Omega)\cap C^{\alpha}(\bar{\Omega}))^3$ satisfies the boundary condition \eqref{1.3-1}, and
		\begin{align}\label{E1}
			\sup\limits_{(x,r)\in \Omega}(\rho+|u|+|v|+|\nabla u|+|\nabla v|)<\infty,
			\quad 
			\rho>0,\text{ and }u>0\text{ in $\Omega$}.
		\end{align}
		Furthermore, assume that
		\begin{align}\label{E2}
			\lim\limits_{x\to \pm\infty}(v,|\nabla v|,\partial_{r}\rho)=0, 
			\quad u\geq \delta\text{ if $r>0$, and $\sqrt{x^2+r^2}\geq R$},
		\end{align}
		for some positive constants $\delta$ and $R$. Then $(\rho,u,v)$ is a solution of the Euler system \eqref{1.3} if and only if $(\rho,u,v)$ solves the continuity equation $\eqref{1.3}_{1}$, \eqref{1.4-0} and \eqref{1.4-1}.
	\end{proposition}
	
	\smallskip
	
	In view of the continuity equation $\eqref{1.3}_{1}$, we can introduce the stream function $\psi(x,r)$, which is determined by
	$$
	\partial_{x}\psi=-r\rho v,\quad \partial_{r}\psi=r\rho u.
	$$
	It follows from the far-field condition \eqref{1.2} that the stream function in the upstream can be described as
	\begin{align}\label{1.6}
		\psi_{\infty}(r)=:\psi(-\infty,r)=\rho_{\infty}\int_{0}^{r}u_{\infty}(s)s\,{\rm d}s\quad \text{for $r\geq 0$}.
	\end{align}
	Similar to \cite{CDXX-2016}, it is easy to check the stream line passing through the point $(x,r)$ with $r>f(x)$ will never touch the boundary $\Gamma$ and go to infinity in the $r-$direction with finite $x$ provided \eqref{E1}--\eqref{E2} hold. Therefore, for any point $(x,r)$ with $r>f(x)$, along the stream line, we can find a position $\kappa>0$ at the upstream such that
	$$
	\psi(x,r)=\psi_{\infty}(\kappa)=\rho_{\infty}\int_{0}^{\kappa}u_{\infty}(s)s\,{\rm d}s,
	$$
	which, together with the fact that $\rho_{\infty}u_{\infty}>0$, yields that $\psi$ is a increasing function of $\kappa$. Thus, we can represent $\kappa=\kappa(\psi;\rho_{\infty})$ as a function $\psi$, and one has
	\begin{align}\label{1.6-1}
		\psi=\rho_{\infty}\int_{0}^{\kappa(\psi;\rho_{\infty})}u_{\infty}(s)s\,{\rm d}s.
	\end{align}
	See Figure \ref{fig3}.
	
	\begin{figure}[htbp]
		\centering
		\begin{tikzpicture}[scale=2.5]
			\draw[dashed] (-1,0) -- (0,0);
			\draw[dashed] (1,0) -- (2,0);
			
			\draw[thick,black,domain=0.001:0.999,samples=200]
			plot (\x,{8*exp(-1/((\x+0.07)*(1.07-\x)))});
			\draw[thick,black,domain=0.001:0.999,samples=200]
			plot (\x,{-8*exp(-1/((\x+0.07)*(1.07-\x)))});
			
			\foreach \x in {0.5}
			{
				\pgfmathsetmacro{\rad}{-8*exp(-1/((\x+0.07)*(1.07-\x)))}
				\pgfmathsetmacro{\xr}{0.05}
				
				\begin{scope}
					\clip (\x-\xr,-\rad) rectangle (\x, \rad);
					\draw[dashed]
					(\x,0) ellipse[x radius=\xr, y radius=\rad];
				\end{scope}
				
				\begin{scope}
					\clip (\x, -\rad) rectangle (\x+\xr, \rad);
					\draw[thick]
					(\x,0) ellipse[x radius=\xr, y radius=\rad];
				\end{scope}
			}
			
			
			\draw[->,thick] (-1,0.4) to[out=20, in=188] (2,0.5);
			
			\fill (-1,0.4) circle (0.5pt);
			
			\fill (2,0.5) circle (0.5pt);
			
			\node[above] at (-1,0.45) {$(-\infty,\kappa(\psi;\rho_{\infty}))$};
			
			\node[above] at (2,0.5) {$(x,r)$}; 
			
		\end{tikzpicture}
		\caption{The schematic of topological structure of streamline}
		\label{fig3}
	\end{figure}
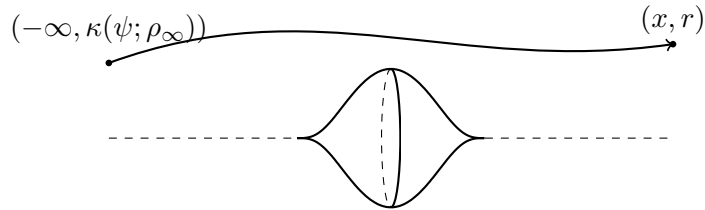
	
	Moreover, one has from \eqref{1.4-0} and \eqref{1.4-1} that $B$ and $\omega$ are both constants along the streamline, that is, the Bernoulli function $B$ in \eqref{1.4} can be determined by
	\begin{align}\label{1.7}
		\frac{1}{2}q^2+h(\rho)=\frac{1}{2}\frac{|\nabla \psi|^2}{\rho^2 r^2}+h(\rho)=\frac{1}{2}u_{\infty}^2(\kappa(\psi;\rho_{\infty}))+h(\rho_{\infty})=B(\psi;\rho_{\infty})
	\end{align}
	and $\omega$ can be determined by
	\begin{align}\label{1.9}
		\frac{\omega}{r\rho}=\frac{-u_{\infty}'(\kappa(\psi;\rho_{0}))}{\rho_{\infty}\kappa(\psi;\rho_{0})}\Rightarrow \omega=-\frac{r\rho u_{\infty}'(\kappa(\psi;\rho_{\infty}))}{\rho_{\infty}\kappa(\psi;\rho_{\infty})}.
	\end{align}
	Set $\mathcal{M}(\psi)=\frac{|\nabla \psi|^2}{r^2}$, then the relationship \eqref{1.7} becomes
	\begin{align}\label{1.8}
		\frac{\mathcal{M}}{2\rho^2}+h(\rho)=B(\psi;\rho_{\infty}).
	\end{align}
	
	\smallskip
	
	For any given $\mathfrak{s}$, it is clear that there exist unique $\rho_{*}(\mathfrak{s})$ and $\rho^{*}(\mathfrak{s})$ such that
	\begin{align}\label{1.8-1}
		\frac{1}{2}\rho_{*}^{\gamma-1}(\mathfrak{s})+h(\rho_{*}(\mathfrak{s}))=\mathfrak{s}\quad \text{ and }\quad h(\rho^{*}(\mathfrak{s}))=\mathfrak{s}.
	\end{align}
	Define
	\begin{align}\label{1.8-2}
		\Sigma(\mathfrak{s})=\rho_{*}^{\gamma+1}(\mathfrak{s}),
	\end{align}
	which is the maximum value of $\mathcal{M}$ associated with sonic state in the subsonic branch. In fact, it is easy to check from 
	\begin{align}\label{1.8-3}
		\frac{\mathcal{M}}{2\rho^2}+h(\rho)=\mathfrak{s}
	\end{align}
	that for any given $\mathfrak{s}$, $\mathcal{M}$ is a strictly decreasing function of $\rho$ for $\rho\in [\rho_{*}(\mathfrak{s}),\rho^{*}(\mathfrak{s})]$. Therefore, for given $\mathfrak{s}$ and $\mathcal{M}$, one can find a unique $\rho\in [\rho_{*}(\mathfrak{s}),\rho^{*}(\mathfrak{s})]$ satisfying \eqref{1.8-3}, see Figure \ref{fig2}.
	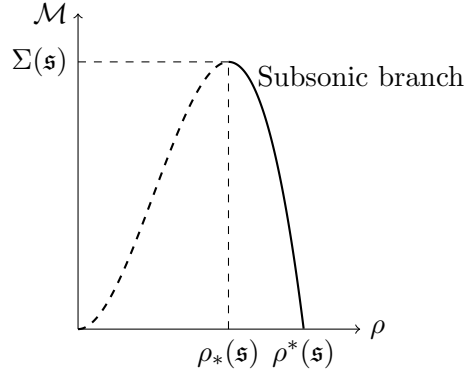
\begin{figure}[htbp]
		\centering
		\begin{tikzpicture}[scale=1.5]
			
			\draw[->] (0,0) -- (2.5,0) node[right] {$\rho$};
			\draw[->] (0,0) -- (0,2.8) node[left] {$\mathcal{M}$};
			
			\def\xcrit{4/3}
			
			\draw[thick,dashed,domain=0:\xcrit,samples=100]
			plot (\x,{4*(\x)^2 - 2*(\x)^3});
			
			\draw[thick,domain=\xcrit:2,samples=100]
			plot (\x,{4*(\x)^2 - 2*(\x)^3});
			
			\draw[dashed] (\xcrit,0) -- (\xcrit,{4*(\xcrit)^2 - 2*(\xcrit)^3});
			\draw[dashed] (0,{4*(\xcrit)^2 - 2*(\xcrit)^3}) -- (\xcrit,{4*(\xcrit)^2 - 2*(\xcrit)^3});

			\node[below] at (4/3,0) {$\rho_{*}(\mathfrak{s})$};
			
			\node[below] at (2,0) {$\rho^{*}(\mathfrak{s})$};
			
			\node[left] at (0,64/27) {$\Sigma(\mathfrak{s})$};
			
			\node[right] at (3/2,{4*(3/2)^2-2*(3/2)^3}) {\text{Subsonic branch}};
			
		\end{tikzpicture}
		\caption{The relationship between $\cal{M}$ and $\rho$}
		\label{fig2}
	\end{figure}
	
	In particular, it follows from \eqref{1.8} that the density is a function of $\mathcal{M}$ and $\psi$, denoted as
	$$
	\rho=H(\mathcal{M},B(\psi;\rho_{\infty}))=:H(\mathcal{M},\psi;\rho_{\infty}).
	$$
	Then, we obtain from \eqref{1.9} that
	\begin{align}\label{1.9-1}
		\omega=-\frac{rH(\mathcal{M},\psi;\rho_{\infty})u_{\infty}'(\kappa(\psi;\rho_{\infty}))}{\rho_{\infty}\kappa(\psi;\rho_{\infty})}.
	\end{align}
	Noting the identity
	$$
	\operatorname{div}\big(\frac{\nabla\psi}{\rho r}\big)=-\omega,
	$$
	we have
	\begin{align}\label{1.10}
		\operatorname{div}\Big(\frac{\nabla\psi}{rH(\mathcal{M},\psi;\rho_{\infty})}\Big)=\frac{rH(\mathcal{M},\psi;\rho_{\infty}) u_{\infty}'(\kappa(\psi;\rho_{\infty}))}{\rho_{\infty}\kappa(\psi;\rho_{\infty})}.
	\end{align}
	A direct calculation from \eqref{1.6-1} shows that
	$$
	\frac{d}{d\psi}\kappa(\psi;\rho_{\infty})=\frac{1}{\rho_{\infty}\kappa(\psi;\rho_{\infty})u_{\infty}(\kappa(\psi;\rho_{\infty}))}.
	$$
	Denoting $\Theta(\psi;\rho_{\infty})=u_{\infty}(\kappa(\psi;\rho_{\infty}))$, we can rewrite \eqref{1.10} as
	\begin{align}\label{1.11}
		\operatorname{div}\Big(\frac{\nabla\psi}{rH(\mathcal{M},\psi;\rho_{\infty})}\Big)=rH(\mathcal{M},\psi;\rho_{\infty})\Theta(\psi;\rho_{\infty})\Theta'(\psi;\rho_{\infty}).
	\end{align}
	Here $\Theta'(\psi;\rho_{\infty})=\frac{d}{d\psi}\Theta(\psi;\rho_{\infty})$.
	
	The boundary condition \eqref{1.3-1} implies that $\psi$ is constant on the boundary $\Gamma$. Noting  $\psi_{\infty}(0)=0$, we impose the following Dirichlet boundary condition for \eqref{1.11}:
	\begin{align}\label{1.12}
		\psi=0\qquad \text{on $\Gamma$}.
	\end{align}
	By direct calculations, one has
	\begin{align}\label{1.12-1}
		H_{1}=:\frac{\partial H(\mathcal{M},\psi;\rho_{\infty})}{\partial \mathcal{M}}=-\frac{H}{2(H^{\gamma+1}-\mathcal{M})},\quad H_{2}=:\frac{\partial H(\mathcal{M},\psi;\rho_{\infty})}{\partial \psi}=\frac{\Theta\Theta'H^3}{H^{\gamma+1}-\mathcal{M}}.	
	\end{align}
	Then we can rewrite \eqref{1.11}--\eqref{1.12} in the following non-divergence form:
	\begin{align}\label{1.13}
		\left\{
		\begin{aligned}
			&\Big(\Big(1-\frac{\mathcal{M}}{H^{\gamma+1}}\Big)\delta_{ij}+\frac{1}{H^{\gamma+1}}\frac{\psi_{i}}{r}\frac{\psi_{j}}{r}\Big)\partial_{ij}\psi-\frac{\psi_{2}}{r}=r^2\Theta\Theta'H^{2},\\
			&\psi=0\qquad \text{on }\Gamma=\{(x,r)\,|\,r=f(x)\},
		\end{aligned}
		\right.
	\end{align}
	where we have denoted $(\psi_{1},\psi_{2})=(\partial_{1}\psi,\partial_{2}\psi)=:(\psi_{x},\psi_{r})$ and have used the Einstein summation convention. Let
	$$
	A_{ij}=\Big(1-\frac{\mathcal{M}}{H^{\gamma+1}}\Big)\delta_{ij}+\frac{1}{H^{\gamma+1}}\frac{\psi_{i}}{r}\frac{\psi_{j}}{r}.
	$$
	Then \eqref{1.13} becomes
	\begin{align}\label{1.13-1}
		A_{ij}\partial_{ij}\psi=r^2\Theta\Theta'H^2+\frac{\psi_{2}}{r}.
	\end{align}
	A straightforward computation gives that the eigenvalues of the matrix $A_{ij}$ are
	$$
	\lambda=\frac{H^{\gamma+1}-\mathcal{M}}{H^{\gamma+1}},\quad \Lambda=1.
	$$
	Thus, the ratio $\frac{\Lambda}{\lambda}=\frac{H^{\gamma+1}}{H^{\gamma+1}-\mathcal{M}}$ is not uniformly bounded as the flow reaches the sonic state, which implies the elliptic equation \eqref{1.13-1} will degenerate at the sonic state. 
	
	\section{Existence of Subsonic Solution of \eqref{1.13} in a Nozzle}
	The problem \eqref{1.13} for subsonic flows is a Dirichlet problem for a quasilinear elliptic equation in a domain $\Omega$. Since $\Omega$ is unbounded in both $x$ and $r$ directions, the stream function may become an unbounded function, which is one of the main differences from the problem of subsonic flows in infinitely long axisymmetric nozzles \cite{DD-2011,DD-2016}. Motivated by \cite{CDXX-2016}, to overcome this difficulty, we first establish the approximated problems in some infinitely long nozzles and then obtain the existence of the subsonic solution in $\Omega$ by some uniform estimates for the approximated solutions.
	
	Let $J=\sup f(x)$. For any given $L\in \mathbb{N}$ satisfying $L>J$, we denote
	$$
	\Omega_{L}=\{(x,r)\,|\,x\in \R,\,f(x)<r<L\}\text{ and }\partial\Omega_{L}=\Gamma\cup \Gamma_{L},
	$$
	with
	$$
	\Gamma_{L}=\{(x,r)\,|\,r=L\}.
	$$
	Noting that $u_{\infty}'(r)$ may not vanish on $\Gamma_{L}$, we need to make some truncation for $u_{\infty}(r)$. Let
	\begin{align}\label{2.1}
		g_{L}(r)=\left\{
		\begin{aligned}
			&u_{\infty}'(r),\qquad\qquad\quad\,\,\,\,\,\,\text{if }r\leq L-1,\\
			&(L-r)u_{\infty}'(L-1),\quad \text{if }L-1<r\leq L,
		\end{aligned}
		\right.
	\end{align}
	and 
	$$
	u_{\infty,L}(r)=u_{\infty}(0)+\int_{0}^{r}g_{L}(s)\,{\rm d}s.
	$$ 
	It is easy to check that for $r\in [0,L]$,
	$$
	u_{\infty,L}(r)\geq u_{\infty}(L-1)+\frac{u_{\infty}'(L-1)}{2}.
	$$
	If $L$ is sufficiently large, we obtain from $\lim\limits_{r\to \infty}u_{\infty}(r)=\bar{u}$ that $u_{\infty,L}\geq \frac{\bar{u}}{2}>0$ for all $r\in [0,L]$. Furthermore, $u_{\infty,L}(r)$ satisfies $u_{\infty,L}'(L)=g_{L}(L)=0$ and
	\begin{align}\label{2.3}
		u_{\infty,L}''(r)=g_{L}'(r)=\left\{
		\begin{aligned}
			&u_{\infty}''(r),\qquad\quad\quad\, \text{if }r\leq L-1,\\
			&-u_{\infty}'(L-1),\quad \text{if }L-1\leq r\leq L.
		\end{aligned}
		\right.
	\end{align}
	Noting $\eqref{1.3-4}$ and $u_{\infty}'(r)\leq 0$, we have $u_{\infty,L}''(r)r\geq u_{\infty,L}'(r)$ for $r\in [0,L]$.
	
	Let $\kappa_{L}(\psi;\rho_{\infty})$ be determined by
	\begin{align}\label{2.4}
		\psi=\rho_{\infty}\int_{0}^{\kappa_{L}(\psi;\rho_{\infty})}u_{\infty,L}(s)s\,{\rm d}s.
	\end{align}
	Denote
	\begin{align*}
		&m_{L}=\rho_{\infty}\int_{0}^{L}u_{\infty,L}(s)s\,{\rm d}s,\quad \theta_{L}(\psi;\rho_{\infty})=u_{\infty,L}(\kappa_{L}(\psi;\rho_{\infty})),
	\end{align*}
	and
	\begin{align}\label{2.6}
		&B_{L}(\psi;\rho_{\infty})=h(\rho_{\infty})+\frac{1}{2}\theta_{L}^2(\psi;\rho_{\infty}),\,\,H_{L}(\mathcal{M},\psi;\rho_{\infty})=H(\mathcal{M},B_{L}(\psi;\rho_{\infty})).
	\end{align}
	Then we consider the problem
	\begin{align}\label{2.7}
		\left\{\begin{aligned}
			&\operatorname{div}\Big(\frac{\nabla \psi}{rH_{L}(\mathcal{M},\psi;\rho_{\infty})}\Big)=rH_{L}(\mathcal{M},\psi;\rho_{\infty})\theta_{L}(\psi;\rho_{\infty})\theta_{L}'(\psi;\rho_{\infty}),\\
			&\psi=0\quad \text{on }\Gamma,\quad \psi=m_{L}\quad \text{on }\Gamma_{L}.
		\end{aligned}
		\right.
	\end{align}
	
	By direct calculations, we can rewrite \eqref{2.9} in the non-divergence form:
	\begin{align}\label{2.9}
		\left\{
		\begin{aligned}
			&\Big(\Big(1-\frac{\mathcal{M}}{H_{L}^{\gamma+1}}\Big)\delta_{ij}+\frac{1}{H_{L}^{\gamma+1}}\frac{\psi_{i}}{r}\frac{\psi_{j}}{r}\Big)\partial_{ij}\psi-\frac{\psi_{2}}{r}=r^2\theta_{L}\theta_{L}'{H}_{L}^{2},\\
			&\psi=0\quad \text{on }\Gamma,\quad \psi=m_{L}\quad \text{on }\Gamma_{L}.
		\end{aligned}
		\right.
	\end{align}
	To establish the existence of problem \eqref{2.9}, we need to construct a sequence of auxiliary regular problems and establish some corresponding uniform estimates to approximate \eqref{2.9}.
	
	\smallskip
	
	\textbf{Step 1. Extension of $\theta_{L}$ and $H_{L}$.} First note that the function $H_{L}(\mathcal{M},\psi;\rho_{\infty})$ is not well defined when $\mathcal{M}$ and $\psi$ are larger than some values, we need to make some extensions on $\psi$.
	
	We first notice that
	$$
	\theta_{L}'(\psi;\rho_{\infty})=\frac{u_{\infty,L}'(\kappa_{L}(\psi;\rho_{\infty}))}{\rho_{\infty}\kappa_{L}(\psi;\rho_{\infty})u_{\infty,L}(\kappa_{L}(\psi;\rho_{\infty}))},
	$$
	which, together with the fact $u_{\infty}'(r)\leq 0$, implies that
	\begin{align}\label{2.10}
		\theta_{L}'(\psi;\rho_{\infty})\leq 0,\quad \theta_{L}'(0;\rho_{\infty})\leq 0\quad \text{ and }\quad \theta_{L}'(m_{L};\rho_{\infty})=0.
	\end{align}
	Then we extend $\theta_{L}$ to $F_{L}$ as follows:
	\begin{align}\label{2.11}
		F_{L}(s;\rho_{\infty})=\left\{
		\begin{aligned}
			&\theta_{L}(s;\rho_{\infty}),\qquad\qquad \qquad \qquad\qquad  \text{if }s\in [0,m_{L}],\\
			&\theta_{L}(m_{L};\rho_{\infty}),\qquad\qquad \qquad \qquad\quad \text{if }s>m_{L},\\
			&\theta_{L}(0;\rho_{\infty})+\theta_{L}'(0;\rho_{\infty})\big(s+\frac{s^2}{2}\big),\quad \text{if }-1\leq s<0,\\
			&\theta_{L}(0;\rho_{\infty})-\frac{\theta_{L}'(0;\rho_{\infty})}{2},\qquad\qquad\,\,\,\,  \text{if }s<-1.
		\end{aligned}
		\right.	
	\end{align}
	It is easy to check that $F_{L}$ is a $C^2$ function on $\R$. Denote
	\begin{align}\label{2.12}
		\check{B}_{L}(\psi;\rho_{\infty})=h(\rho_{\infty})+\frac{1}{2}F_{L}^2(\psi;\rho_{\infty}).
	\end{align}
	
	\textbf{Step 2. Subsonic truncation.} In order to deal with the possible degeneracy at sonic state in \eqref{2.9}, we need to truncate the associated equation near the sonic state. For any $\varepsilon_{0}\in (0,\frac{1}{4})$, let $\chi_{0}(s)$ be a smooth increasing function satisfying
	\begin{align}\label{2.13}
		\chi_{0}(s)=\left\{
		\begin{aligned}
			&s,\qquad\quad\,\,\,\,\,\text{if }s\leq 1-2\varepsilon_{0},\\
			&1-\frac{3}{2}\varepsilon_{0},\quad \text{if }s\geq 1-\varepsilon_{0}.
		\end{aligned}
		\right.
	\end{align}
	Define
	$$
	\check{\mathcal{M}}(\mathcal{M},\psi;\rho_{\infty})=\chi_{0}^2\Big(\frac{\sqrt{\mathcal{M}}}{H^{\frac{\gamma+1}{2}}(\mathcal{M},\check{B}_{L}(\psi;\rho_{\infty}))}\Big)H^{\gamma+1}(\mathcal{M},\check{B}_{L}(\psi;\rho_{\infty})).
	$$
	The corresponding truncated density $\check{H}_{L}(\mathcal{M},\psi;\rho_{\infty})$ is determined by 
	$$
	\frac{\check{\mathcal{M}}}{2\check{H}_{L}^2(\mathcal{M},\psi;\rho_{\infty})}+\frac{1}{\gamma-1}\check{H}_{L}^{\gamma-1}(\mathcal{M},\psi;\rho_{\infty})=\check{B}_{L}(\psi;\rho_{\infty}),
	$$
	that is,
	$$
	\check{H}_{L}(\mathcal{M},\psi;\rho_{\infty})=H(\check{\mathcal{M}},\check{B}_{L}(\psi;\rho_{\infty})).
	$$
	
	{\textbf{Step 3. Truncation of singularity on the axis}}. Due to the singularity at $r=0$, we consider the following approximate truncated problem
	\begin{align}\label{2.9-1}
		\left\{
		\begin{aligned}
			&\left(\left(1-\chi_{0}^2\left(\frac{|\nabla\psi|}{(r+k)(\check{H}_{L}^{k})^{\frac{\gamma+1}{2}}}\right)\right)\delta_{ij}+\chi_{0}\left(\frac{\partial_{i}\psi}{(r+k)(\check{H}_{L}^{k})^{\frac{\gamma+1}{2}}}\right)\chi_{0}\left(\frac{\partial_{j}\psi}{(r+k)(\check{H}_{L}^{k})^{\frac{\gamma+1}{2}}}\right)\right)\partial_{ij}\psi\\
			&\qquad =(r+k)^2F_{L}F_{L}'\big(\check{H}_{L}^{k}\big)^2+\frac{\psi_{2}}{r+k},\\
			&\psi=0\quad \text{on }\Gamma,\quad \psi=m_{L}\quad \text{on }\Gamma_{L},
		\end{aligned}
		\right.
	\end{align}
	for any $k>0$. Here
	$$
	\check{H}_{L}^{k}(\mathcal{M}_{k},\psi;\rho_{\infty})=H(\check{\mathcal{M}}_{k},\check{B}_{L}(\psi;\rho_{\infty})),
	$$
	with
	$$
	\check{\mathcal{M}}_{k}(\mathcal{M}_{k},\psi;\rho_{\infty})=\chi_{0}^2\Big(\frac{\sqrt{\mathcal{M}}_{k}}{H^{\frac{\gamma+1}{2}}(\mathcal{M}_{k},\check{B}_{L}(\psi;\rho_{\infty}))}\Big)H^{\frac{\gamma+1}{2}}(\mathcal{M}_{k},\check{B}_{L}(\psi;\rho_{\infty})),\,\,\text{and }\mathcal{M}_{k}=\frac{|\nabla\psi|^2}{(r+k)^2}.
	$$ 
	We rewrite \eqref{2.9-1} as 
	\begin{align}\label{2.9-2}
		\check{A}_{ij}^{k}\partial_{ij}\psi=\check{F}^{k}+\check{G}^{k},
	\end{align}
	where
	\begin{align*}
		\check{A}_{ij}^{k}=\left(\left(1-\chi_{0}^2\left(\frac{|\nabla\psi|}{(r+k)(\check{H}_{L}^{k})^{\frac{\gamma+1}{2}}}\right)\right)\delta_{ij}+\chi_{0}\left(\frac{\partial_{i}\psi}{(r+k)(\check{H}_{L}^{k})^{\frac{\gamma+1}{2}}}\right)\chi_{0}\left(\frac{\partial_{j}\psi}{(r+k)(\check{H}_{L}^{k})^{\frac{\gamma+1}{2}}}\right)\right),
	\end{align*}
	and
	\begin{align*}
		\check{F}^{k}=(r+k)^2F_{L}F_{L}'\big(\check{H}_{L}^{k}\big)^2,\quad \check{G}^{k}=\frac{\psi_{2}}{r+k}.
	\end{align*}
	It is easy to check \eqref{2.9-2} is a uniformly elliptic equation for any $(x,r)\in \Omega_{L}$, and $\check{F}^{k}+\check{G}^{k}$ satisfy the structural conditions in \cite[Theorem 12.5]{GT}.

\begin{lemma}[Existence of \eqref{2.9-1}]\label{lem2.1}
For any $\rho_{\infty}>\rho_{\infty}^{*}$, there exists a solution $\psi_{L}^{k}\in C^{2,\beta}(\Omega_{L})\cap C^{1,\beta}(\bar{\Omega}_{L})$ for some $\beta<\nu$ to the problem \eqref{2.9-1} satisfying
\begin{align}\label{2.14}
	0\leq \psi_{L}^{k}\leq m_{L}\quad \text{in }\overline{\Omega}_{L}.
\end{align}
Furthermore, there exists a constant $\rho_{\infty,L}^{*}\in (\rho_{\infty}^{*},\infty)$ independent of $k$ such that if $\rho_{\infty}>{\rho}_{\infty,L}^{*}$, it holds that
\begin{align}\label{2.15}
	\sup_{(x,r)\in \overline{\Omega}_{L}}\frac{|\nabla\psi_{L}^{k}|}{(r+k)(\check{H}_{L}^{k})^{\frac{\gamma+1}{2}}(\mathcal{M},\psi_{L}^{k};\rho_{\infty})}< 1-2\varepsilon_{0}.
\end{align}
	\end{lemma}

	\noindent\textbf{Proof}: Motivated by \cite{CDXX-2016,DD-2011,DD-2016}, to deal with the unboundedness of the domain $\Omega_{L}$ in $x$-direction, we need to truncate the domain and establish some uniform estimates to approximate \eqref{2.9-1}. We divide the proof into several steps.
	
	\textbf{Step 1}. For any given integer $N>0$, using similar arguments as in \cite[Appendix]{XX-2007}, we can construct a domain  $\Omega_{L,N}$ satisfying
	\begin{itemize}
\item [(\rmnum{1})] $\{(x,r)\,|\,(x,r)\in \Omega_{L}, -N<x<N\}\subset \Omega_{L,N}\subset \{(x,r)\,|\,(x,r)\in \Omega_{L},-4N<x<4N\}$;
\item [(\rmnum{2})] $\Omega_{L,N}\in C^{2,\lambda}$ for some constant $0<\lambda<\nu$ satisfies the uniform interior sphere condition with uniform radius $\mathfrak{r}_{0}$ for all $N>N_{0}$ with some integer $N_{0}$ sufficiently large.
\end{itemize}
Noting that $\check{F}^{k}+\check{G}^{k}$ satisfies the structural conditions required in \cite[Theorem 12.5]{GT}, then it follows from \cite[Theorem 12.5]{GT} that there is a solution $\psi_{L,N}^{k}\in C^{2,\mu}(\Omega_{L,N})\cap C^{1,\mu}(\overline{\Omega}_{L,N})$ with $0<\mu<\lambda$ for the following truncated problem:
\begin{align}\label{2.22}
\left\{
\begin{aligned}
	&\check{A}_{ij}^{k}\partial_{ij}\psi=\check{F}^{k}+\check{G}^{k},\quad \text{in }\Omega_{L,N},\\
	&\psi=0\,\, \text{on }\Gamma\cap \partial\Omega_{L,N},\,\, \psi=\frac{m_{L}}{\int_{0}^{L}(s+k)u_{\infty,L}(s)\,{\rm d}s}\int_{0}^{r}(s+k)u_{\infty,L}(s)\,{\rm d}s\,\, \text{on }\partial\Omega_{L,N}\backslash\Gamma.
\end{aligned}
\right.
\end{align}

\textbf{Step 2}. It follows from \eqref{2.10}--\eqref{2.11} that $\check{F}^{k}\leq 0$ if $\psi\leq 0$, and $\check{F}^{k}=0$ if $\psi\geq m_{L}$. Then the maximum principle implies that
$$
0\leq \psi_{L,N}^{k}\leq m_{L}\quad \text{in }\overline{\Omega}_{L,N}.
$$

\textbf{Step 3}. By using similar arguments as in \cite[Section 3.5]{DD-2011} and \cite[Section 2.2]{DX-2014}, we have the following uniform estimates for fixed $L>0$ and $k>0$,
\begin{align}\label{2.22-0}
\|\psi_{L,K}^{k}\|_{1,\mu;\Omega_{L,N}}\leq \mathcal{C}(\frac{\Lambda_{k}}{\lambda_{k}},f,L)\Big(1+m_{L}+\frac{1}{\lambda_{k}}|\check{F}^{k}|_{0}\Big),
\end{align}
and
\begin{align}\label{2.22-01}
\|\psi_{L,K}^{k}\|_{2,\mu,\Omega_{L,N}}^{-1-\mu}\leq \mathcal{C}\Big(\frac{\Lambda_{k}}{\lambda_{k}},f,L,m_{L},\frac{1}{\lambda_{k}}|\check{F}^{k}|_{0}\Big),
\end{align}
for any $K\geq 4N$, where $\lambda_{k}$ and $\Lambda_{k}$ are the eigenvalues of the matrix $[\check{A}_{ij}^{k}]_{2\times 2}$. 

\textbf{Step 4}. We first claim that there exists a constant $\bar{\rho}_{\infty,L,k}^{*}>\rho_{\infty}^{*}$ depending on $k$ and $L$ such that if $\rho_{\infty}\geq \bar{\rho}_{\infty,L,k}^{*}$, one has
\begin{align}\label{2.15-01}
\sup_{(x,r)\in \overline{\Omega}_{L,N}}\frac{|\nabla \psi_{L,N}^{k}|}{(r+k)(\check{H}_{L}^{k})^{\frac{\gamma+1}{2}}\Big(\frac{|\nabla\psi_{L,N}^{k}|^2}{(r+k)^2},\psi_{L,N}^{k};\rho_{\infty}\Big)}< 1-2\varepsilon_{0}.
\end{align}
In fact, we first notice from \eqref{1.8-1} and \eqref{2.12} that $$\rho_{*}(\check{B}_{L}(\psi_{L,N}^{k};\rho_{\infty}))\sim O(\rho_{\infty}) \;\; \text{ and } \;\; \rho^{*}(\check{B}_{L}(\psi_{L,N}^{k};\rho_{\infty}))\sim O(\rho_{\infty})$$ for any $\rho_{\infty}>\rho_{\infty}^{*}$, which, together with the fact $H(\mathcal{M},\psi;\rho_{\infty})\in (\rho_{*},\rho^{*})$ for any $\mathcal{M}$ and $\psi$, implies that
\begin{align}\label{2.22-2}
\check{H}_{L}^{k}\Big(\frac{|\nabla\psi_{L,N}^{k}|^2}{(r+k)^2},\psi_{L,N}^{k};\rho_{\infty}\Big)=H\Big(\frac{|\nabla\psi_{L,N}^{k}|^2}{(r+k)^2},\check{B}_{L}(\psi_{L,N}^{k};\rho_{\infty})\Big)\sim O(\rho_{\infty}).
\end{align}
Therefore, one has
\begin{align*}
\check{F}^{k}=(r+k)^2F_{L}F_{L}'(\check{H}_{L}^{k})^2=(r+k)^2\frac{u_{\infty}'(\kappa(\psi_{L,N}^{k};\rho_{\infty}))}{\rho_{\infty}\kappa(\psi_{L,N}^{k};\rho_{\infty})}(\check{H}_{L}^{k})^2\leq C_{L}\rho_{\infty},
\end{align*}
which yields
\begin{align*}
\frac{|\check{F}^{k}|}{\lambda_{k}}\leq C_{L}\rho_{\infty},
\end{align*}
due to the fact $\lambda_{k}\geq C_{\varepsilon_{0}}>0$ for some positive constant $C_{\varepsilon_{0}}$ depending only on $\varepsilon_{0}>0$.
Moreover, a direct computation gives that for any fixed $\varepsilon_{0}$,
$$
\frac{\Lambda_{k}}{\lambda_{k}}\sim O(1),\qquad m_{L}=\rho_{\infty}\int_{0}^{L}su_{\infty,L}(s)\,{\rm d}s\leq C_{L}\rho_{\infty}.
$$
Therefore, we obtain from \eqref{2.22-0} and \eqref{2.22-2} that
\begin{align*}
\frac{|\nabla \psi_{L,N}^{k}|}{(r+k)\check{H}_{L}^{\frac{\gamma+1}{2}}\Big(\frac{|\nabla\psi_{L,N}^{k}|^2}{(r+k)^2},\psi_{L}^{k};\rho_{\infty}\Big)}\leq \frac{|\nabla \psi_{L,N}^{k}|}{k\check{H}_{L}^{\frac{\gamma+1}{2}}\Big(\frac{|\nabla\psi_{L,N}^{k}|^2}{(r+k)^2},\psi_{L}^{k};\rho_{\infty}\Big)}\leq C_{L,k}\rho_{\infty}^{\frac{1-\gamma}{2}},
\end{align*}
where $C_{L,k}$ depends on $L$ and $k$. So, taking $\bar{\rho}_{\infty,L,k}^{*}$ large enough such that $C_{L,k}\bar{\rho}_{\infty,L,k}^{*}< 1-2\varepsilon_{0}$, we conclude \eqref{2.15-01} for $\rho_{\infty}\geq \bar{\rho}_{\infty,L,k}^{*}$.

\textbf{Step 5}. We are going to prove that the constant $\bar{\rho}_{\infty,L,k}^{*}$ in Step 4 can be chosen independent of $k$. To achieve this, we first need to establish uniform estimates with respect to $k$ near the axis. Different from the ones in \cite{DD-2011,DD-2016}, $\check{F}^{k}\leq 0$ due to $\theta_{L}'\leq 0$ by our assumption on $u_{\infty}(r)$, which results in a slightly more complicated construction on the barrier function. Indeed, we set
$$
\tilde{\psi}(r)=\rho_{\infty}(\mathcal{A}_{L}-\mathfrak{B}_{L}r)(r+k)^2,
$$
where $\mathcal{A}_{L}$ and $\mathcal{B}_{L}$ are some suitably large constants depending on $L$, which will be chosen later. A straightforward computation shows that
\begin{align*}
\tilde{\psi}_{x}=0,\quad \tilde{\psi}_{r}=2\rho_{\infty}(\mathcal{A}_{L}-\mathcal{B}_{L}r)(r+k)-\rho_{\infty}\mathcal{B}_{L}(r+k)^2.
\end{align*}
Then, we have
\begin{align*}
\frac{|\nabla \tilde{\psi}|^2}{(r+k)^2}=\rho_{\infty}^2\big[2(\mathcal{A}_{L}-\mathcal{B}_{L}r)-\mathcal{B}_{L}(r+k)\big]^2\leq C_{L}\rho_{\infty}^2,\quad \frac{|\nabla \tilde{\psi}|^2}{(r+k)^2\check{H}_{L}(|\frac{\tilde{\psi}}{r+k}|^2,\tilde{\psi};\rho_{\infty})}\leq C_{L}\rho_{\infty}^{1-\gamma},
\end{align*}
which implies that
$$
\sup_{(x,r)\in \overline{\Omega}_{L}}\frac{|\nabla \tilde{\psi}|}{(r+k)(\check{H}_{L}^{k})^{\frac{\gamma+1}{2}}(|\frac{\nabla \tilde{\psi}}{r+k}|^2,\tilde{\psi};\rho_{\infty})}< 1-2\varepsilon_{0},
$$
provided that $\rho_{\infty}>\bar{\rho}_{\infty,L}^{*}$ for some large constant $\bar{\rho}_{\infty,L}^{*}$ depending only $L$ and $\varepsilon_{0}$. Then at the point where $\nabla \psi_{L,N}^{k}=\nabla\tilde{\psi}$, we have
$$
\frac{\nabla \psi_{L,N}^{k}}{(r+k)(\check{H}_{L}^{k})^{\frac{\gamma+1}{2}}\Big(\frac{|\nabla\psi_{L,N}^{k}|^2}{(r+k)^2},\psi_{L};\rho_{\infty}\Big)}< 1-2\varepsilon_{0}.
$$
Then, by direct calculations, one has that, at the point where $\nabla\psi_{L,N}^{k}=\nabla\tilde{\psi}$,
\begin{align}\label{2.23}
\check{A}_{ij}^{k}\partial_{ij}\tilde{\psi}-\frac{\tilde{\psi}_{r}}{r+k}=\tilde{\psi}_{rr}-\frac{\tilde{\psi}_{r}}{r+k}=-\rho_{\infty}\mathcal{B}_{L}(3r+k).
\end{align}
For the solution $\psi_{L,N}^{k}$ of the problem \eqref{2.22}, a direct calculation shows that
\begin{align*}
&(r+k)^2F_{L}(\psi_{L,N}^{k};\rho_{\infty})F_{L}'(\psi_{L,N}^{k};\rho_{\infty})(\check{H}_{L}^{k})^2\Big(\frac{|\nabla\psi_{L,N}^{k}|^2}{(r+k)^2},\psi_{L,N}^{k};\rho_{\infty}\Big)\nonumber\\
&\leq \frac{(r+k)^2u_{\infty}'(\kappa_{L}(\psi_{L,N}^{k};\rho_{\infty}))(\check{H}_{L}^{k})^2\big(\frac{|\nabla\psi_{L,N}^{k}|^2}{(r+k)^2},\psi_{L,N}^{k};\rho_{\infty}\big)}{\rho_{\infty}\kappa_{L}(\psi_{L,N}^{k};\rho_{\infty})}\leq C_{L}\rho_{\infty}(r+k).
\end{align*}
So, taking $\mathcal{B}_{L}$ large enough, we obtain from \eqref{2.22} and \eqref{2.23} that, at the point where $\nabla\psi_{L,N}^{k}=\nabla \tilde{\psi}$,
$$
\check{A}_{ij}^{k}\partial_{ij}(\psi_{L,N}^{k}-\tilde{\psi})-\frac{\partial_{r}(\psi_{L,N}^{k}-\tilde{\psi})}{r+k}\geq 0.
$$
On the other hand, taking $\mathcal{A}_{L}$ large enough such that $\tilde{\psi}(r)\geq \psi_{L,N}^{k}$ on $\partial\Omega_{L,N}$, we obtain from the comparison principle that
\begin{align}\label{2.27}
\psi_{L,N}^{k}\leq \tilde{\psi}\leq C_{L}\rho_{\infty}(r+k)^2\quad \text{in }\overline{\Omega}_{L,N},
\end{align}
provided $\rho_{\infty}>\bar{\rho}_{\infty,L}^{*}$.

For any fixed point $(x_{0},r_{0})\in \R\times (0,\delta_{0})$ for some small $\delta_{0}>0$, we set
$$
\psi_{0}(x,r)=\frac{1}{\xi^2}\psi_{L,N}^{k}(x_{0}+x\xi,r_{0}+r\xi+k),\quad \text{with $\xi=\frac{r_{0}+2k}{2}$},
$$
which is well defined in $B_{1}(0,0)$. Moreover, a direct calculation yields that
$$
\frac{\nabla \psi_{L,N}^{k}(x_{0}+x\xi,r_{0}+r\xi+k)}{r_{0}+r\xi+2k}=\frac{\nabla \psi_{0}}{2+r}.
$$
Moreover, we obtain from the Step 2 and Step 4 that $\psi_{L,N}^{k}$ satisfies
\begin{align}\label{2.29-0}
&\operatorname{div}\Bigg(\frac{\nabla \psi_{L,N}^{k}}{(r+k)H_{L}\Big(\frac{|\nabla\psi_{L,N}^{k}|^2}{(r+k)^2},\psi_{L,N}^{k};\rho_{\infty}\Big)}\Bigg)\nonumber\\
&=(r+k)\theta_{L}(\psi_{L,N}^{k})\theta_{L}'(\psi_{L,N}^{k})H_{L}\Big(\frac{|\nabla\psi_{L,N}^{k}|^2}{(r+k)^2},\psi_{L,N}^{k};\rho_{\infty}\Big).
\end{align}
provided that $\rho_{\infty}>\max\{\bar{\rho}_{\infty,L,k}^{*},\bar{\rho}_{\infty,L}^{*}\}$. So, $\psi_{0}$ satisfies
\begin{align}\label{2.29}
&\operatorname{div}\Bigg(\frac{\nabla\psi_{0}}{(2+r)H_{L}\Big(|\frac{\nabla \psi_{0}}{2+r}|^2,\xi^2\psi_{0};\rho_{\infty}\Big)}\Bigg)\nonumber\\
&=(r_{0}+r\xi+k)\theta_{L}(\xi^2\psi_{0})\theta_{L}'(\xi^2\psi_{0})H_{L}\Big(\big\vert\frac{\nabla \psi_{0}}{2+r}\big\vert^2,\xi^2\psi_{0};\rho_{\infty}\Big).
\end{align}
Due to \eqref{2.27}, for $\rho_{\infty}\geq\max\{\bar{\rho}_{\infty,L,k}^{*},\bar{\rho}_{\infty,L}^{*}\}$, we have
\begin{align*}
\psi_{0}(x,r)=\frac{4}{(r_{0}+2k)^2}\psi_{L,N}^{k}(x_{0}+x\xi,r_{0}+\xi r+k)\leq C_{L}\rho_{\infty}.
\end{align*}
Then, applying Moser's iteration \cite[Theorem 8.35]{GT} to \eqref{2.29}, we have
$$
|\nabla \psi_{0}(0,0)|\leq C_{L}\rho_{\infty}.
$$
In particular,
$$
\Big\vert\frac{\nabla \psi_{L,N}^{k}}{r+k}\Big\vert\leq |\nabla \psi_{0}(0,0)|\leq C\rho_{\infty}\quad \text{for $r<\delta_{0}$}.
$$
On the other hand, for any $r>\delta_{0}$, we obtain from \eqref{2.22-0} that
$$
\frac{|\nabla \psi_{L,N}^{k}|}{r+k}\leq C_{L,\delta_{0}}\rho_{\infty}.
$$
In conclusion, for $\rho_{\infty}\geq \max\{\bar{\rho}_{\infty,L,K}^{*},\bar{\rho}_{\infty,L}^{*}\}$, we have
\begin{align}\label{2.29-2}
\frac{|\nabla \psi_{L,N}^{k}|}{r+k}\leq C_{L}\rho_{\infty}\quad \text{in $\overline{\Omega}_{L,N}$},
\end{align}
where $C_{L}$ depends only on $L$ and $\varepsilon_{0}$. 

\textbf{Step 6}. Now, given $\rho_{\infty}\in (\rho_{\infty}^{*},\infty)$, let $\mathcal{S}_{L}^{k}(\rho_{\infty})$ be the set of all solutions $\psi_{L,N}^{k}$ of the problem \eqref{2.22} associated with $\rho_{\infty}$. Define
\begin{align}\label{2.29-3}
\mathcal{Q}_{L}^{k}(\rho_{\infty})=\sup_{\psi_{L,N}^{k}\in \mathcal{S}_{L}^{k}(\rho_{\infty})}\sup_{(x,r)\in \overline{\Omega}_{L,N}}\frac{|\nabla \psi_{L,N}^{k}|}{(r+k)(\check{H}_{L}^{k})^{\frac{\gamma+1}{2}}\Big(\frac{|\nabla\psi_{L,N}^{k}|^2}{(r+k)^2},\psi_{L,N}^{k},\rho_{\infty}\Big)}.
\end{align}
We denote
\begin{align}\label{2.29-4}
\rho_{\infty,L,k}^{*}=\operatorname{inf}\{s\,|\,\mathcal{Q}_{L}^{k}<1-2\varepsilon_{0}\text{ for any $\rho_{\infty}>s$}\}.
\end{align}
It is clear that $\rho_{\infty,L,k}^{*}\leq \max\{\bar{\rho}_{\infty,L}^{*},\bar{\rho}_{\infty,L,k}^{*}\}$.

We claim that if $\rho_{\infty,L,k}^{*}\geq 2\bar{\rho}_{\infty,L}^{*}$, then $\mathcal{Q}_{L}^{k}(\rho_{\infty,L,k}^{*})=1-2\varepsilon_{0}$. Indeed, first, it follows from the continuous dependence on the parameter $\rho_{\infty}$ for the solution of uniformly elliptic equations \eqref{2.22} that $\mathcal{Q}_{L}^{k}(\rho_{\infty,L,k}^{*})\leq 1-2\varepsilon_{0}$. Second, if $\mathcal{Q}_{L}^{k}(\rho_{\infty,L,k}^{*})< 1-2\varepsilon_{0}$ and $\rho_{\infty,L,k}^{*}\geq 2\bar{\rho}_{\infty,L}^{*}$, we obtain from \eqref{2.29-2} that $\mathcal{Q}_{L}^{k}(\rho_{\infty,L,k}^{*})\leq C_{L}\rho_{\infty}^{\frac{1-\gamma}{2}}$. Therefore, there exists a $\delta>0$ such that $\mathcal{Q}_{L}^{k}(s)\leq 1-2\varepsilon_{0}$ for any $s\in (\rho_{\infty,L,k}^{*}-\delta,\rho_{\infty,L,k}^{*})$, which contradicts with the definition of $\rho_{\infty,L,k}^{*}$. 

Thus, if $\rho_{\infty,L,k}^{*}\geq 2\bar{\rho}_{\infty,L}^{*}$, there exists a solution $\psi_{L,N}^{k}$ of \eqref{2.22} associated with $\rho_{\infty}=\rho_{\infty,L,k}^{*}$ satisfying 
\begin{align*}
\frac{|\nabla \psi_{L,N}^{k}|}{(r+k)(\check{H}_{L}^{k})^{\frac{\gamma+1}{2}}\Big(\frac{|\nabla\psi_{L,N}^{k}|^2}{(r+k)^2},\psi_{L,N}^{k};\rho_{\infty}\Big)}=1-2\varepsilon_{0}.
\end{align*}
Noting \eqref{2.29-2}, we have
\begin{align*}
1-2\varepsilon_{0}&= \frac{|\nabla \psi_{L,N}^{k}|}{(r+k)(\check{H}_{L}^{k})^{\frac{\gamma+1}{2}}\Big(\frac{|\nabla\psi_{L,N}^{k}|^2}{(r+k)^2},\psi_{L,N}^{k};\rho_{\infty}\Big)}\nonumber\\
&\leq \frac{\sup_{(x,r)\in \overline{\Omega}_{L,N}} \frac{|\nabla \psi_{L}^{k}|}{r+k}}{\inf_{(x,r)\in \overline{\Omega}_{L,N}}(\check{H}_{L}^{k})^{\frac{\gamma+1}{2}}\Big(\frac{|\nabla\psi_{L,N}^{k}|^2}{(r+k)^2},\psi_{L,N}^{k};\rho_{\infty}\Big)}\leq C_{L}(\rho_{\infty,L,k}^{*})^{\frac{1-\gamma}{2}},
\end{align*}
which implies that
$$
\rho_{\infty,L,k}^{*}\leq C_{L,\varepsilon_{0}},
$$
where $C_{L,\varepsilon_{0}}$ depends only on $L$ and $\varepsilon_{0}$ and is independent of $k$. Then choosing
$$
\rho_{\infty,L}^{*}=\max\{2\bar{\rho}_{\infty,L}^{*},C_{L,\varepsilon_{0}}\}.
$$
It is clear that $\rho_{\infty,L}^{*}$ is independent of $k$, and if $\rho_{\infty}>\rho_{\infty,L}^{*}$, there exists a solution $\psi_{L,N}^{k}$ to \eqref{2.22} satisfying \eqref{2.15-01}. In particular, $\psi_{L,N}^{k}$ will solve \eqref{2.29-0}.

\textbf{Step 7}. In view of \eqref{2.22-0}--\eqref{2.22-01}, it follows from the Arzela-Ascoli lemma and diagonal procedure that there exists a sub-sequence $\{\psi_{L,K_{l}}^{k}\}_{l=1}^{\infty}$ for some fixed $k$ and $l$, such that $\psi_{L,K_{l}}^{k}\to \psi_{L}^{k}$ in $C^{2,\beta}(\Omega_{L})\cap C^{1,\beta}(\bar{\Omega}_{L})$ with $\beta<\mu<\nu$. And $\psi_{L}^{k}$ solves \eqref{2.9-1} and satisfies \eqref{2.14} and
\begin{align*}
\|\psi_{L}^{k}\|_{1,\beta;\Omega_{L}}\leq C(\frac{\Lambda_{k}}{\lambda_{k}},L,f)\Big(1+m_{L}+\frac{1}{\lambda_{k}}|\check{F}^{k}|_{0}\Big).
\end{align*}
Moreover, taking $N\to \infty$, we obtain from \eqref{2.27} and \eqref{2.15-01} that for $\rho_{\infty}>\rho_{\infty,L}^{*}$,
\begin{align*}
\psi_{L}^{k}\leq C_{L}\rho_{\infty}(r+k)^2.
\end{align*}
and $\psi_{L}^{k}$ satisfies \eqref{2.15}. Hence, the proof of Lemma \ref{lem2.1} is complete. $\hfill\square$

\begin{lemma}\label{lem2.3}
Let $\rho_{\infty,L}^{*}$ be the constant found in Lemma \ref{lem2.1}, then it holds that
\begin{align}\label{2.15-0}
	\partial_{r}\psi_{L}^{k}>0\quad \text{for all $(x,r)\in \overline{\Omega}_{L}$},
\end{align}
which means the subsonic flow $\psi_{L}^{k}$ is never degenerate in $\overline{\Omega}_{L}$.
\end{lemma}

\noindent\textbf{Proof}. The non-degenerate condition \eqref{2.15-0} is more complicated to prove for axisymmetric Euler flows compared with the ones for 2D Euler flows \cite{CDXX-2016,XX-2010-2}. Indeed, different from \cite{CDXX-2016,XX-2010-2}, the coefficients in the equation of stream function depend on $r$, and the associated terms are out of control after applying $\partial_{r}$ on the coefficients and then multiplying by $\partial_{r}\psi_{L}^{k}$. So, we are not able to prove \eqref{2.15-0} by considering directly the equation of $\partial_{r}\psi_{L}^{k}$ and using the energy estimates as in \cite{CDXX-2016,XX-2010-2}. Motivated by \cite{DWX-2018}, to overcome this difficulty, we start from an observation on a potential flow case with coefficients truncated near the axis, and then use a continuity argument and maximum principle to achieve \eqref{2.15-0}. We divide the proof into four steps.

\textbf{Step 1}. We first consider the irrotational case. For the potential flow, the Bernoulli's function is a constant $B_{L}(\psi)\equiv B_{L}(0)=:B_{\infty}$, \textit{i.e.}, $u_{\infty,L}(r)\equiv u_{\infty}(0)$ is a constant. Then we consider the solution of the equation:
\begin{align}\label{P1}
\left\{
\begin{aligned}
	&\operatorname{div}\Big(\frac{\nabla \tilde{\psi}_{L,N}^{k}}{(r+k)H_{L}\big(\big\vert\frac{\nabla \tilde{\psi}_{L,N}^{k}}{r+k}\big\vert^2,\tilde{\psi}_{L,N}^{k};\rho_{\infty}\big)}\Big)=0,\\
	&\tilde{\psi}_{L,N}^{k}=0\,\,\text{on }\Gamma,\quad \tilde{\psi}_{L,N}^{k}=\frac{\rho_{\infty}u_{\infty}(0)\int_{0}^{L}s\,{\rm d}s}{\int_{0}^{L}(s+k)\,{\rm d}s}\int_{0}^{r}(s+k)\,{\rm d}s\quad \text{on }\partial\Omega_{L,N}\backslash \Gamma.
\end{aligned}
\right.
\end{align}
The existence of a subsonic solution of \eqref{P1} for $N$ large enough when $\rho_{\infty}\geq \rho_{\infty,L}^{*}$ can be proved by using the same arguments as in the proof of Lemma \ref{lem2.1}. Moreover, it holds that
$$
0\leq \tilde{\psi}_{L,N}^{k}\leq \frac{1}{2}\rho_{\infty}u_{\infty}(0)L^2.
$$
Noting that $\tilde{\psi}_{L,N}^{k}$ attains the minimum on the lower boundary $\Gamma$, and the maximum on the upper boundary $\Gamma_{L}$, we obtain from the Hopf point lemma that
$$
\partial_{r}\tilde{\psi}_{L,N}^{k}>0\quad \text{on $\Gamma$ and $\Gamma_{L}$}.
$$
Furthermore, on the left and right boundaries, we obtain from the boundary conditions in \eqref{P1}
that
$$
\partial_{r}\tilde{\psi}_{L,N}^{k}(x,r)\vert_{\partial\Omega_{L,N}\backslash (\Gamma\cup \Gamma_{L})}=\rho_{\infty}u_{\infty}(0)(r+k)\frac{\int_{0}^{L}s\,{\rm d}s}{\int_{0}^{L}(s+k)\,{\rm d}s}>0.
$$
So, in conclusion, we have
\begin{align}\label{E1-1}
\partial_{r}\tilde{\psi}_{L,N}^{k}>0\quad \text{on }\partial\Omega_{L,N}.
\end{align}
For the potential flow \eqref{P1}, to show that $\partial_{r}\tilde{\psi}_{L,N}^{k}>0$ in $\Omega_{L,N}$, we shall consider the associated potential function rather than the stream function $\tilde{\psi}_{L,N}^{k}$. The advantage is that we only need to take $\partial_{x}$ on both sides of the equation of the potential function to avoid the difficulty that the coefficients depend on $r$. To this end, we denote
$$
u_{N}^{k}=\frac{\partial_{r}\tilde{\psi}_{L,N}^{k}}{(r+k)H_{L}\big(\big\vert\frac{\nabla \tilde{\psi}_{L,N}^{k}}{r+k}\big\vert^2,\tilde{\psi}_{L,N}^{k};\rho_{\infty}\big)},\quad v_{N}^{k}=-\frac{\partial_{x}\tilde{\psi}_{L,N}^{k}}{(r+k)H_{L}\big(\big\vert\frac{\nabla \tilde{\psi}_{L,N}^{k}}{r+k}\big\vert^2,\tilde{\psi}_{L,N}^{k};\rho_{\infty}\big)},
$$
and
$$
\rho_{N}^{k}=H_{L}\big(\big\vert\frac{\nabla \tilde{\psi}_{L,N}^{k}}{r+k}\big\vert^2,\tilde{\psi}_{L,N}^{k};\rho_{\infty}\big).
$$
Thus, it is clear that $(\rho^{k},u_{N}^{k},v_{N}^{k})$ satisfies following Euler equations:
\begin{align}\label{P2}
\left\{
\begin{aligned}
	&((r+k)\rho^{k}u_{N}^{k})_{x}+((r+k)\rho^{k}v_{N}^{k})_{r}=0,\\
	&(u_{N}^{k})_{r}-(v_{N}^{k})_{x}=0,
\end{aligned}
\right.
\end{align}
with the Bernoulli law:
\begin{align}\label{E3}
\frac{1}{2}[(u_{N}^{k})^2+(v_{N}^{k})^2]+\frac{1}{\gamma-1}(\rho_{N}^{k})^{\gamma-1}=\frac{1}{2}(u_{\infty}(0))^2+\frac{1}{\gamma-1}\rho_{\infty}^{\gamma-1}=:B_{\infty}.
\end{align}
Now, we introduce the potential function $\Phi_{N}^{k}$, which satisfies
$$
\partial_{x}\Psi_{N}^{k}=u_{N}^{k},\quad \partial_{r}\Psi_{N}^{k}=v_{N}^{k}.
$$
Then it follows from \eqref{E3} that the density $\rho_{N}^{k}$ can be written as a function of $\Phi_{N}^{k}$:
$$
\rho_{N}^{k}=H_{L}(|\nabla \Phi_{N}^{k}|^2,B_{\infty}).
$$

We further rewrite $\eqref{P2}_{1}$ as
\begin{align}\label{E4}
\operatorname{div}\big(g(|\nabla \Phi_{N}^{k}|^2,r+k,B_{\infty})\nabla \Phi_{N}^{k}\big)=0,
\end{align}
where
$$
g(|\nabla \Phi_{N}^{k}|^2,r+k,B_{\infty})=(r+k)\rho_{N}^{k}=(r+k)H_{L}(|\nabla \Phi_{N}^{k}|^2,B_{\infty}).
$$
By direct calculations, we can rewrite \eqref{E4} in the following non-divergence form:
\begin{align}\label{E5}
\mathfrak{a}_{ij}(\nabla \Phi_{N}^{k},B_{\infty})\partial_{ij}\Phi_{N}^{k}+\frac{1}{r+k}\partial_{r}\Phi_{N}^{k}=0,
\end{align}
where
$$
\mathfrak{a}_{ij}=\delta_{ij}-\frac{\partial_{i}\Phi_{N}^{k}\partial_{j}\Phi_{N}^{k}}{H_{L}^{\gamma-1}(|\nabla \Phi_{N}^{k}|^2,B_{\infty})}.
$$
Applying $\partial_{x}$ to \eqref{E5} to get $u_{N}^{k}=\partial_{x}\Phi_{N}^{k}$ satisfies
\begin{align}\label{E4-1}
\mathfrak{a}_{ij}\partial_{ij}u_{N}^{k}+D_{p_{m}}\mathfrak{a}_{ij}\partial_{ij}\Phi_{N}^{k}\partial_{m}u_{N}^{k}+\frac{1}{r+k}\partial_{r}u_{N}^{k}=0,
\end{align}
where $D_{p_{m}}\mathfrak{a}_{ij}$ means the partial derivative of $\mathfrak{a}_{ij}(\nabla \Phi_{N}^{k},B_{\infty})$ with respect to $\partial_{m}\Phi_{N}^{k}$
for $m=1,2$. We obtain from \eqref{E1-1} that
$$
u_{N}^{k}=\Phi_{x}>0\quad \text{on }\partial\Omega_{L,N}.
$$
Applying the maximum principle to \eqref{E4-1}, we have
$$
u_{N}^{k}>0\quad \text{in $\overline{\Omega}_{L,N}$},
$$
which yields
\begin{align}\label{E7}
\partial_{r}\tilde{\psi}_{L,N}^{k}>0\quad \text{in $\overline{\Omega}_{L,N}$}.
\end{align}
That is, $\partial_{r}\tilde{\psi}_{L,N}^{k}$ will never degenerate including the axis. That is why we need to introduce the truncated coefficient $(r+k)$ near the axis even for the irrotational case. In particular, 
\begin{align}\label{E7-1}
\min_{(x,r)\in \overline{\Omega}_{L,N}}\partial_{r}\tilde{\psi}_{L,N}^{k}>0.
\end{align}

\textbf{Step 2}. Next, let us consider the case of Euler flow when the upstream flow is a small perturbation of the potential flow. For potential flow, we denote by $\rho=H_{L}\big(\big\vert\frac{\nabla \tilde{\psi}_{L,N}^{k}}{r}\big\vert^2,B_{\infty};\rho_{\infty}\big)$ the density, which is determined by
\begin{align}\label{B1}
\frac{|\nabla\tilde{\psi}_{L,N}^{k}|^2}{2(r+k)^2\rho^2}+\frac{1}{\gamma-1}\rho^{\gamma-1}=\frac{1}{2}(u_{\infty}(0))^2+\frac{1}{\gamma-1}\rho_{\infty}^{\gamma-1}=B_{\infty}.
\end{align}

For the Euler flow, we consider the subsonic solution $\psi_{L,N}^{k}$ to \eqref{2.29-0}, that is, $\psi_{L,N}^{k}$ satisfies
\begin{align}\label{2.22-00}
\left\{
\begin{aligned}
	&\operatorname{div}\left(\frac{\nabla \psi_{L,N}^{k}}{(r+k)H_{L}\big(\frac{|\nabla \psi_{L,N}^{k}|^2}{(r+k)^2},B_{L}(\psi_{L,N}^{k});\rho_{\infty}\big)}\right)\\
	&\quad =(r+k)\theta_{L}(\psi_{L,N}^{k})\theta_{L}'(\psi_{L,N}^{k})H_{L}\big(\frac{|\nabla \psi_{L,N}^{k}|^2}{(r+k)^2},B_{L}(\psi_{L,N}^{k});\rho_{\infty}\big),\\
	&\psi_{L,N}^{k}=0\quad \text{on }\Gamma\cap \partial\Omega_{L,N},\\
	&\psi_{L,N}^{k}=\frac{m_{L}}{\int_{0}^{L}u_{\infty,L}(s)(s+k)\,{\rm d}s}\int_{0}^{r}u_{\infty,L}(s)(s+k)\,{\rm d}s\quad \text{on }\partial\Omega_{L,N}\backslash \Gamma,
\end{aligned}
\right.
\end{align}
and we denote the corresponding density by $\rho={H}_{L}\big(\big\vert\frac{\nabla \psi_{L,N}^{k}}{r}\big\vert^2,B_{L}(\psi_{L,N}^{k});\rho_{\infty}\big)$, which is determined by
\begin{align}\label{B2}
\frac{|\nabla \psi_{L,N}^{k}|^2}{2(r+k)^2\rho^2}+\frac{1}{\gamma-1}\rho^{\gamma-1}=\frac{1}{2}\big(\theta_{L}(\kappa_{L}(\psi_{L,N}^{k}))\big)^2+\frac{1}{\gamma-1}\rho_{\infty}^{\gamma-1}:=B_{L}(\psi_{L,N}^{k}).
\end{align}

Let $\phi=\psi_{L,N}^{k}-\tilde{\psi}_{L,N}^{k}$, taking the difference between \eqref{P1} and \eqref{2.22-00} to get
\begin{align}\label{Q3}
&\operatorname{div}\Big(\frac{\nabla\psi_{L,N}^{k}}{(r+k)H_{L}\big(\big\vert\frac{\nabla\psi_{L,N}^{k}}{r+k}\big\vert^2,B_{L}(\psi_{L,N}^{k});\rho_{\infty}\big)}-\frac{\nabla\tilde{\psi}_{L,N}^{k}}{(r+k)H_{L}\big(\big\vert\frac{\nabla\tilde{\psi}_{L,N}^{k}}{r+k}\big\vert^2,B_{\infty};\rho_{\infty}\big)}\Big)\nonumber\\
&=(r+k)\theta_{L}(\psi_{L,N}^{k})\theta_{L}'(\psi_{L,N}^{k})H_{L}\big(\big\vert\frac{\nabla\psi_{L,N}^{k}}{r+k}\big\vert^2,B_{L}(\psi_{L,N}^{k});\rho_{\infty}\big).	
\end{align}
We split the terms in parentheses on the left-hand side as follows
\begin{align}\label{Q4}
&\frac{\nabla\psi_{L,N}^{k}}{(r+k)H_{L}\big(\big\vert\frac{\nabla\psi_{L,N}^{k}}{r+k}\big\vert^2,B_{L}(\psi_{L,N}^{k});\rho_{\infty}\big)}-\frac{\nabla\tilde{\psi}_{L,N}^{k}}{(r+k)H_{L}\big(\big\vert\frac{\nabla\tilde{\psi}_{L,N}^{k}}{r+k}\big\vert^2,B_{\infty};\rho_{\infty}\big)}\nonumber\\
&=\frac{\nabla\psi_{L,N}^{k}}{(r+k)H_{L}\big(\big\vert\frac{\nabla\psi_{L,N}^{k}}{r+k}\big\vert^2,B_{L}(\psi_{L,N}^{k});\rho_{\infty}\big)}-\frac{\nabla\tilde{\psi}_{L,N}^{k}}{(r+k)H_{L}\big(\big\vert\frac{\nabla\tilde{\psi}_{L,N}^{k}}{r+k}\big\vert^2,B_{L}(\tilde{\psi}_{L,N}^{k});\rho_{\infty}\big)}\nonumber\\
&\quad +\frac{\nabla\tilde{\psi}_{L,N}^{k}}{(r+k)H_{L}\big(\big\vert\frac{\nabla\tilde{\psi}_{L,N}^{k}}{r+k}\big\vert^2,B_{L}(\tilde{\psi}_{L,N}^{k});\rho_{\infty}\big)}-\frac{\nabla\tilde{\psi}_{L,N}^{k}}{(r+k)H_{L}\big(\big\vert\frac{\nabla\tilde{\psi}_{L,N}^{k}}{r+k}\big\vert^2,B_{\infty};\rho_{\infty}\big)}\nonumber\\
&=\sum\limits_{i=1}^2J_{i}.
\end{align}
For $J_{1}$, denoting $\psi_{L,N,t}^{k}=t\psi_{L,N}^{k}+(1-t)\tilde{\psi}_{L,N}^{k}$, $\mathcal{M}_{L,N,t}=\frac{|\nabla \psi_{L,N,t}^{k}|^2}{(r+k)^2}$ and using the fact
$$
f(\nabla \psi_{L,N},\psi_{L,N}^{k})-f(\nabla \tilde{\psi}_{L,N}^{k},\tilde{\psi}_{L,N}^{k})=\int_{0}^{1}
\frac{d}{dt}f(\nabla\psi_{L,N,t}^{k} ,\psi_{L,N,t}^{k})\,{\rm d}t,
$$
we can rewrite $J_{1}$ as
\begin{align}\label{Q4-1}
J_{1}=a_{ij}(\nabla \psi_{L,N,t}^{k},\psi_{L,N,t}^{k})\partial_{j}\phi+b_{i}(\nabla \psi_{L,N,t}^{k},\psi_{L,N,t}^{k})\phi,
\end{align}
where
\begin{align*}
a_{ij}(\nabla \psi_{L,N,t}^{k},\psi_{L,N,t}^{k})&=\int_{0}^{1}\frac{1}{(r+k)H_{L}(\mathcal{M}_{L,N,t}^{k},\psi_{L,N,t}^{k})}\delta_{ij}\,{\rm d}t\\
&\quad -\int_{0}^{1}\frac{1}{(r+k)H_{L}^2(\mathcal{M}_{L,N,t}^{k},\psi_{L,N,t}^{k})}\frac{\partial H_{L}}{\partial\mathcal{M}}\frac{2\partial_{i}\psi_{L,N,t}^{k}\partial_{j}\psi_{L,N,t}^{k}}{(r+k)^2}\,{\rm d}t,\\
b_{i}(\nabla \psi_{L,N,t}^{k},\psi_{L,N,t}^{k})&=-\int_{0}^{1}\frac{\partial H_{L}}{\partial\psi}\frac{\partial_{i}\psi_{L,N,t}^{k}}{(r+k)H_{L}^2(\mathcal{M}_{L,N,t}^{k},\psi_{L,N,t}^{k})}\,{\rm d}t,
\end{align*}
with
$$
\frac{\partial H_{L}}{\partial\mathcal{M}}=-\frac{H_{L}(\mathcal{M}_{L,N,t}^{k},\psi_{L,N,t}^{k})}{2(H_{L}^{\gamma+1}(\mathcal{M}_{L,N,t}^{k},\psi_{L,N,t}^{k})-\mathcal{M}_{L,N,t}^{k})},\quad \frac{\partial H_{L}}{\partial \psi}=\frac{\theta_{L}(\psi_{L,N,t}^{k})\theta_{L}'(\psi_{L,N,t}^{k})H_{L}^3(\mathcal{M}_{L,N,t}^{k},\psi_{L,N,t}^{k})}{H_{L}^{\gamma+1}(\mathcal{M}_{L,N,t}^{k},\psi_{L,N,t}^{k})-\mathcal{M}_{L,N,t}^{k}}.
$$
For $J_{2}$, using the intermediate value theorem, one has
\begin{align}\label{Q5-1}
J_{2}&=\frac{\nabla \tilde{\psi}_{L,N}^{k}}{r+k}\frac{H_{L}(\big\vert\frac{\nabla\tilde{\psi}_{L,N}^{k}}{r+k}\big\vert^2,B_{\infty})-H_{L}(\big\vert\frac{\nabla\tilde{\psi}_{L,N}^{k}}{r+k}\big\vert^2,B_L(\tilde{\psi}_{L,N}^{k})}{H_{L}(\big\vert\frac{\nabla\tilde{\psi}_{L,N}^{k}}{r+k}\big\vert^2,B_{\infty})H_{L}(\big\vert\frac{\nabla\tilde{\psi}_{L,N}^{k}}{r+k}\big\vert^2,B_L(\tilde{\psi}_{L,N}^{k}))}\nonumber\\
&=-\frac{\nabla \tilde{\psi}_{L,N}^{k}}{r+k}\frac{1}{H_{L}(\big\vert\frac{\nabla\tilde{\psi}_{L,N}^{k}}{r+k}\big\vert^2,B_{\infty})H_{L}(\big\vert\frac{\nabla\tilde{\psi}_{L,N}^{k}}{r+k}\big\vert^2,B_L(\tilde{\psi}_{L,N}^{k}))}\frac{\partial H_{L}}{\partial B_{L}}\Big\vert_{(\big\vert\frac{\nabla\tilde{\psi}_{L,N}^{k}}{r+k}\big\vert^2,\xi)}(B_{L}(\tilde{\psi}_{L,N}^{k})-B_{\infty})\nonumber\\
&=:{\bf d}(\nabla \tilde{\psi}_{L,N}^{k},\tilde{\psi}_{L,N}^{k})(B_{L}(\tilde{\psi}_{L,N}^{k})-B_{\infty}).
\end{align}
Substituting \eqref{Q4-1} and \eqref{Q5-1} into \eqref{Q3}, we have
\begin{align}\label{Q3-1}
&\partial_{i}[a_{ij}(\nabla \psi_{L,N,t}^{k},\psi_{L,N,t}^{k})\partial_{j}\phi+b_{i}(\nabla \psi_{L,N,t}^{k},\psi_{L,N,t}^{k})\phi]\nonumber\\
&\quad =-\partial_{i}\big[d_{i}(\nabla \tilde{\psi}_{L,N}^{k},\tilde{\psi}_{L,N}^{k})(B_{L}(\tilde{\psi}_{L,N}^{k})-B_{\infty}\big]\nonumber\\
&\qquad +(r+k)H_{L}(\big\vert\frac{\nabla \psi_{L,N}^{k}}{r+k}\big\vert^2,B_{L}(\psi_{L,N}^{k});\rho_{\infty})B_{L}'(\psi_{L,N}^{k}).
\end{align}

Obviously, \eqref{Q3-1} is a second-order elliptic equation with all coefficients $a_{ij}$, $b_{i}$ and $d_{i}$ being bounded and H\"{o}lder continuous. Moreover, on the boundary, it follows from \eqref{P1} and \eqref{2.22-00} that
$$
|\phi|\leq C\|u_{\infty,L}'\|_{C^{0}[0,L]}.
$$
Thus, using the classical Schauder estimates, we have
\begin{align}\label{E7-2}
\|\phi\|_{C^{2,\alpha}(B_{\mathfrak{r}_{0}}(x,r))}=\|\psi_{L,N}^{k}-\tilde{\psi}_{L,N}^{k}\|_{C^{2,\alpha}(B_{\mathfrak{r}_{0}}(x,r))}\leq C_{L,\mathfrak{r}_{0},k}\|u_{\infty}'\|_{C^{0,1}[0,L]}\leq C_{L,\mathfrak{r}_{0},k}\delta,
\end{align}
provided that $\|u_{\infty,L}'\|_{C^{0,1}[0,L]}\leq \delta$. Here, $\mathfrak{r}_{0}$ is the uniform radius in the uniform interior sphere condition of $\Omega_{L,N}$, and $C_{L,k,\mathfrak{r}_{0}}$ depends only on $L$, $k$ and $\mathfrak{r}_{0}$ and is independent of the position $(x,r)$. Furthermore, it follows from \eqref{E7-1} and \eqref{E7-2} that
\begin{align*}
\partial_{r}\psi_{L,N}^{k}>0\quad \text{in $\overline{\Omega}_{L,N}$}
\end{align*}
provided $\delta<\delta_{*}$ sufficiently small. Here $\delta_{*}$ depends only on $L$, $k$, $\mathfrak{r}_{0}$ and $\min_{(x,r)\in \overline{\Omega}_{L,N}}\partial_{r}\tilde{\psi}_{L,N}^{k}$.

\textbf{Step 3}. In this step, we will show $\partial_{r}\psi_{L,N}^{k}>0$ in $\overline{\Omega}_{L,N}$ for any positive upstream velocity $u_{\infty}(r)$ based modified Bers' method. First, by using similar arguments as in the proof of \eqref{E1-1}, we have
\begin{align}\label{E1-2}
\partial_{r}\psi_{L,N}^{k}>0\quad \text{on $\partial\Omega_{L,N}$}
\end{align}
for any positive upstream velocity $u_{\infty}(r)$. Then, we apply $\partial_{r}$  to $\eqref{2.22-00}_{1}$ to see $U_{L,N}^{k}=\partial_{r}\psi_{L,N}^{k}$ will satisfy
\begin{align}\label{2.42-1}
&a_{ij}\partial_{ij}U_{L,N}^{k}+\partial_{r}\Big(1-\frac{(U_{L,N}^{k})^2}{(r+k)^2H_{L}^{\gamma+1}}\Big)\partial_{11}\psi+\partial_{r}(a_{12}+a_{21})\partial_{x}U_{L,N}^{k}+\partial_{r}\Big(1-\frac{(\pa_{x}\psi_{L,N}^{k})^2}{(r+k)^2H_{L}^{\gamma+1}}\Big)\partial_{2}U_{L,N}^{k}\nonumber\\
&\quad -\Big(\frac{\partial_{r}U_{L,N}^{k}}{r+k}-\frac{U_{L,N}^{k}}{(r+k)^2}\Big)-(r+k)^2\partial_{r}\big(\theta_{L}\theta_{L}'H_{L}^2)=2(r+k)\theta_{L}\theta_{L}'H_{L}^2,
\end{align}
where
$$
a_{ij}=\Big(1-\frac{|\nabla \psi_{L,N}^{k}|^2}{(r+k)^2H_{L}^{\gamma+1}}\Big)\delta_{ij}+\frac{1}{H_{L}^{\gamma+1}}\frac{\pa_{i}\psi_{L,N}^{k}}{r+k}\frac{\pa_{j}\psi_{L,N}^{k}}{r+k}.
$$
A direct calculation shows that \eqref{2.42-1} can be written in following form:
\begin{align}\label{2.43-1}
a_{ij}\partial_{ij}U_{L,N}^{k}+b_{i}\partial_{i}U_{L,N}^{k}+cU_{L,N}^{k}=2(r+k)\theta_{L}\theta_{L}'H_{L}^2\leq 0,
\end{align}
due to the fact $u_{\infty}'(r)\leq 0$.
Applying the strong maximum principle without requiring the sign of zeroth order term $cU_{L,N}^{k}$ to \eqref{2.43-1}, we know if the solution $U_{L,N}^{k}\geq 0$ in $\overline{\Omega}_{L,N}$, then $U_{L,N}^{k}$ cannot attain the minimum value zero in $\Omega_{L,N}$, \textit{i.e.}, $\partial_{r}\psi_{L,N}^{k}>0$ in $\overline{\Omega}_{L,N}$.

Let $\delta_{0}$ be the largest positive number such that if $\delta<\delta_{0}$ and $\|u_{\infty,L}'\|_{C^{0,1}[0,L]}\leq \delta$, then $\partial_{r}\psi_{L,N}^{k}>0$ in $\overline{\Omega}_{L,N}$. Assume there exists a upstream flow $u_{\infty,L}(r)$ satisfying $\|u_{\infty,L}'\|_{C^{0,1}[0,L]}=\delta_{0}$ such that the associated solution $\psi_{L,N}^{k}$ satisfies $\partial_{r}\psi_{L,N}^{k}\leq 0$ at one point $(x_{0},r_{0})$ in $\Omega_{L,N}$. Then, by the definition of $\delta_{0}$, there exists a sequence $u_{\infty,L}^{(n)}\to u_{\infty,L}$ satisfying $\|u_{\infty,L}^{(n)}\|_{C^{1,1}[0,L]}\leq \delta_{0}$, and the associated solution $(\psi_{L,N}^{k,(n)})_{r}>0$ in $\overline{\Omega}_{L,N}$. Then by the continuous dependence on $u_{\infty}$, we know $\psi_{L,N}^{k,(n)}\to \psi_{L,N}^{k}$ in $C^{2,\alpha}(\Omega_{L,N})$ as $n\to \infty$. Therefore $\partial_{r}\psi_{L,N}^{k}\geq 0$ in $\Omega_{L,N}$. Then we obtain from the strong maximum principle and \eqref{E1-2} that $(\psi_{L,N}^{k})_r>0$, which is a contradiction. Hence, \eqref{E1-2} holds for any positive upstream velocity $u_{\infty}\in C^2([0,\infty))$ satisfying $u_{\infty}'(r)\leq 0$.

\textbf{Step 4}. Since Step 7 in the proof of Lemma \ref{lem2.1} shows the uniform boundedness of  $\psi_{L,N}^{k}$ in  $C^{2,\beta}(\Omega_{L})\cap C^{1,\beta}(\overline{\Omega}_{L})$, we know that $\psi_{L,N}^{k}\to \psi_{L}^{k}$ in any compact set $ K\subset\Omega_{L}$ as $N\to \infty$ with
$$
\partial_{r}\psi_{L}^{k}\geq 0\quad \text{for all $(x,r)\in \overline{\Omega}_{L}$}.
$$
Then, applying the maximum principle and the Hopf point lemma on the boundary, we conclude \eqref{2.15-0}. Therefore, the proof of Lemma \ref{lem2.3} is complete. $\hfill\square$

\medskip

As a direct corollary, we have the following existence of a subsonic solution of \eqref{2.9}. 

\begin{lemma}[Existence of subsonic solution to \eqref{2.9}]\label{lem2.2}

There exists a constant $\rho_{\infty,L}^{*}\in (\rho_{\infty}^{*},\infty)$ depending only on $L$ such that if $\rho_{\infty}>\rho_{\infty,L}^{*}$, there exists a solution $\psi_{L}\in C^{2,\beta}(\Omega_{L})\cap C^{1,\beta}(\overline{\Omega}_{L})$ with $\beta\in (0,\nu)$ to the problem \eqref{2.9} satisfying
\begin{align}\label{2.14-1}
	0\leq \psi_{L}\leq m_{L},\quad \text{in }\overline{\Omega}_{L},
\end{align}
and
\begin{align}\label{2.15-1}
	\frac{|\nabla\psi_{L}|}{r{H}^{\frac{\gamma+1}{2}}(\mathcal{M},\psi_{L};\rho_{\infty})}< 1-2\varepsilon_{0}.
\end{align}
Moreover, it holds that
\begin{align}\label{2.14-2}
	\partial_{r}\psi_{L}>0\quad \text{for any $(x,r)\in \overline{\Omega}_{L}\cap \{(x,r)\,|\,r>0\}$.}
\end{align}
\end{lemma}

\noindent\textbf{Proof}. By taking the limit $k\to 0$, \eqref{2.14-1}--\eqref{2.15-1} follow directly from \eqref{2.14}--\eqref{2.15}. We only need to prove \eqref{2.14-2}. First, taking the limit $k\to 0$ in \eqref{2.15-0}, we have
\begin{align}\label{2.41}
\partial_{r}\psi_{L}^{k}\geq 0\quad \text{for any $(x,r)\in \overline{\Omega}_{L}$}.
\end{align}
Noting $\psi_{L}$ attains the maximum value at upper boundary $\Gamma_{L}$, and attains the minimum value at the lower boundary $\{(x,r)\,|r=f(x)\text{ for }x\in (0,1)\}$ we obtain from the Hopf point lemma that
\begin{align}\label{2.41-1}
\partial_{r}\psi_{L}>0\quad \text{at }\Gamma\cup \{(x,r)\,|\,r=f(x)\text{ for }x\in (0,1)\}.
\end{align}

Now, we denote $U=\partial_{r}\psi_{L}$, and apply $\partial_{r}$ to \eqref{2.9} to see that $U$ satisfies following equation:
\begin{align}\label{2.42}
&a_{ij}\partial_{ij}U+\partial_{r}\Big(1-\frac{U^2}{r^2H_{L}^{\gamma+1}}\Big)\partial_{11}\psi+\partial_{r}(a_{12}+a_{21})\partial_{1}U+\partial_{r}\Big(1-\frac{\psi_{1}^2}{r^2H_{L}^{\gamma+1}}\Big)\partial_{2}U\nonumber\\
&\quad -\Big(\frac{\partial_{2}U}{r}-\frac{U}{r^2}\Big)-r^2\partial_{r}\big(\theta_{L}\theta_{L}'H_{L}^2)=2r\theta_{L}\theta_{L}'H_{L}^2,
\end{align}
where
$$
a_{ij}=\Big(1-\frac{|\nabla\psi_{L}|^2}{rH_{L}^{\gamma+1}}\Big)\delta_{ij}+\frac{1}{H_{L}^{\gamma+1}}\frac{\psi_{i}}{r}\frac{\psi_{j}}{r}.
$$
A direct calculation shows that \eqref{2.42} can be written in the following form:
\begin{align}\label{2.43}
a_{ij}\partial_{ij}U+b_{i}\partial_{i}U+cU=2r\theta_{L}\theta_{L}'H_{L}^2\leq 0
\end{align}
due to the fact $u_{\infty}'(r)\leq 0$.
Then, using \eqref{2.41} and applying the strong maximum principle without requiring the sign of $c$ to \eqref{2.43}, we can know that $U$ cannot attain the minimum value zero in $\Omega_{L}$, \textit{i.e.},
\begin{align}\label{2.44}
U=\partial_{r}\psi_{L}>0\quad \text{in }\Omega_{L},
\end{align}
which, together with \eqref{2.41-1}, yields \eqref{2.14-2}. This concludes the proof of Lemma \ref{lem2.2}. $\hfill\square$

\section{Uniform Estimates for Subsonic Flows in Nozzles}

In this section, we will establish several uniform estimates in $L$ to prove that $\rho_{\infty,L}^{*}$ in Lemma \ref{lem2.2} can be taken independent of $L$. Consequently, we can take the limit $L\to \infty$ to establish the existence of problem \eqref{1.13}. Precisely, we shall prove the following proposition.

\begin{proposition}\label{prop4.1}
There exists a $\bar{\rho}_{\infty}\in (\rho_{\infty}^{*},\infty)$ independent of $L$ such that if $\rho_{\infty}>\bar{\rho}_{\infty}$, the problem \eqref{1.13} has a subsonic solution $\psi$ satisfying
\begin{align}\label{2.33}
	0\leq \psi\leq \bar{\psi}(r),\quad |\psi-\bar{\psi}|_{C^{1}(\overline{\Omega})}\leq \mathscr{C}\rho_{\infty}r,\quad \iint_{\Omega}\frac{|\nabla (\psi-\bar{\psi})|^2}{r}\,{\rm d}x{\rm d}r\leq \mathscr{C}\rho_{\infty},
\end{align}
and
\begin{align}\label{2.34}
	\sup_{(x,r)\in \overline{\Omega}}\frac{|\nabla\psi|}{rH^{\frac{\gamma+1}{2}}(\mathcal{M},\psi;\rho_{\infty})}< 1-2\varepsilon_{0},
\end{align}
where
$$
\bar{\psi}(r)=\rho_{\infty}\int_{0}^{r}u_{\infty}(s)s\,{\rm d}s,
$$
and the constant $\mathscr{C}$ depends only on $\max f(x)$, $\max u_{\infty}(r)$ and $\varepsilon_{0}$.
\end{proposition}
The proof of Proposition \ref{prop4.1} is based on the following propositions and lemmas.

\begin{proposition}\label{prop1}
For any $\varepsilon>0$, there exists and $\bar{L}>0$ such that if $L>\bar{L}$, then for any $\rho_{\infty}\geq (1+\varepsilon)\rho_{\infty}^{*}$, there exists a unique triple $(\chi(s),\rho_{1,L},u_{1,L})$ satisfying that
\begin{itemize}
	\item [(1)] $\rho_{1,L}$ is a constant, $\chi$ is an increasing function mapping $[0,L]$ to $[J,L]$, and $u_{1,L}$ is a function on $[J,L]$, where $J=\sup f(x)$;
	\item [(2)] the triple $(\chi(s),\rho_{1,L},u_{1,L})$ satisfies
	\begin{align}\label{3.1}
		0<\rho_{1,L}<\rho_{\infty},\quad u_{1,L}(r)>0\quad \text{and }\max_{J\leq r\leq L}u_{1,L}(r)<\rho_{1,L}^{\frac{\gamma-1}{2}};
	\end{align}
	\item [(3)] for $s\in [0,L]$, $(\chi(s),\rho_{1,L},u_{1,L}(\chi(s)))$ satisfies
	\begin{align}\label{3.2}
		\frac{u_{\infty,L}^2(s)}{2}+h(\rho_{\infty})=\frac{u_{1,L}^2(\chi(s))}{2}+h(\rho_{1,L}),
	\end{align}
	and
	\begin{align}\label{3.3}
		\rho_{\infty}\int_{0}^{s}u_{\infty,L}(s)s\,{\rm d}s=\rho_{1,L}\int_{J}^{\chi(s)}u_{1,L}(s)s\,{\rm d}s.
	\end{align}
	Furthermore, let
	\begin{align}\label{3.4}
		\bar{\psi}(r)=\rho_{\infty}\int_{0}^{r}u_{\infty,L}(s)s\,{\rm d}s\quad \text{and}\quad \hat{\psi}(r)=\rho_{1,L}\int_{J}^{r}u_{1,L}(s)s\,{\rm d}s.
	\end{align}
	Then for $r\in [J,L]$, it holds that
	\begin{align}\label{3.5}
		0\leq \bar{\psi}(r)-\hat{\psi}(r)\leq C\rho_{\infty},
	\end{align}
	where $C$ is a uniform constant depending only on $J$ and $\max u_{\infty}$.
\end{itemize}
\end{proposition}

\noindent\textbf{Proof}. The proof is similar to \cite[Proposition 3.3]{CDXX-2016}. For the reader's convenience, we sketch the proof here.

Let $\underline{\varrho}_{1,L}$ and $\overline{\varrho}_{1,L}$ be the constants determined by
\begin{align}\label{3.6}
\frac{1}{2}\underline{\rho}_{1,L}^{\gamma-1}+h(\underline{\rho}_{1,L})=\frac{1}{2}\max u_{0,L}^2(r)+h(\rho_{\infty}),\quad h(\overline{\rho}_{1,L})=\frac{1}{2}\min u_{0,L}^2(r)+h(\rho_{\infty}).	
\end{align}
If $\rho_{\infty}>\rho_{\infty}^{*}$, then it is easy to see that $\underline{\varrho}_{1,L}<\rho_{\infty}<\overline{\varrho}_{1,L}$. Furthermore, denoting 
$$
D(s;\rho)=2(h(\rho_{\infty})-h(\rho))+u_{0,L}^2(s),
$$
then for any $\rho\in (\underline{\rho}_{1,L},\overline{\rho}_{1,L})$, one has
$$
\min_{s\in [0,L]}D(s;\rho)>0\text{ and } \max_{s\in [0,L]}\sqrt{D(s;\rho)}<\rho^{\frac{\gamma-1}{2}}\quad \text{for }s\in [0,1].
$$	
Differentiating \eqref{3.3} with respect to $s$ and using \eqref{3.2} yields that
\begin{align}\label{3.7}
\frac{1}{2}\frac{d\chi^2(s)}{ds}=\frac{\rho_{0,L}u_{0,L}(s)s}{\rho_{1,L}\sqrt{D(s;\rho_{1,L})}}\quad \text{for }s\in [0,L]\quad \text{and }\chi(0)^2=J^2.
\end{align}
Hence, it suffices to show that there exists an $\bar{L}>0$ such that if $L>\bar{L}$, then for any $\rho_{0}\in ((1+\varepsilon)\rho_{0}^{*},\infty)$, there exists a unique $\rho_{1,L}\in (\underline{\rho}_{1,L},\overline{\rho}_{1,L})$ such that
\begin{align}\label{3.8}
\int_{0}^{L}\frac{\rho_{\infty}u_{0,L}(s)s}{\rho_{1,L}\sqrt{D(s;\rho_{1,L})}}\,{\rm d}s=\frac{1}{2}(L^2-J^2).
\end{align}
Once $\rho_{1,L}$ is determined, then existence of  $\chi(s)$ follows from \eqref{3.7} and $u_{1,L}(\chi(s))=\sqrt{D(s;\rho_{1,L})}$.

Let 
$$
G(\rho)=\int_{0}^{L}\frac{\rho_{\infty}u_{0,L}(s)}{\rho\sqrt{D(s;\rho)}}\,{\rm d}s.
$$
A direct computation yields that
\begin{align}\label{3.9}
G'(\rho)=\int_{0}^{L}\frac{\rho_{\infty}u_{0,L}(s)s}{\rho^2D^{\frac{3}{2}}(s;\rho)}\big(\rho^{\gamma-1}-D(s;\rho)\big)\,{\rm d}s,
\end{align}
which is always positive for $\rho\in (\underline{\varrho}_{1,L},\overline{\varrho}_{1,L})$. Thus $G(\rho)$ is a strict increasing function on $(\underline{\varrho}_{1,L},\overline{\varrho}_{1,L})$, which implies the range of $G(\rho)$ is $(G(\underline{\varrho}_{1,L}),G(\overline{\varrho}_{1,L}))$. We shall prove that $\frac{1}{2}(L^2-J^2)$ lies in the interval $(G(\underline{\varrho}_{1,L}),G(\overline{\varrho}_{1,L}))$ if $L$ is sufficiently large.

First, it is clear that $G(\rho_{\infty})=\frac{1}{2}L^2>\frac{1}{2}(L^2-J^2)$. Next, it follows from the definition of $\underline{\varrho}_{1,L}$ in \eqref{3.6} that
\begin{align}\label{3.10}
G(\underline{\varrho}_{1,L})&=\int_{0}^{L}\frac{\rho_{\infty}u_{0,L}(s)s}{\underline{\varrho}_{1,L}\sqrt{\underline{\rho}_{1,L}^{\gamma-1}+u_{0,L}^2(s)-\max u_{0,L}^2(s)}}\,{\rm d}s\nonumber\\
&=\int_{0}^{L}\frac{\rho_{\infty}u_{0,L}(s)s}{\big(\frac{\gamma-1}{\gamma+1}\max u_{0,L}^2(s)+\frac{2}{\gamma+1}\rho_{\infty}^{\gamma-1}\big)^{\frac{1}{\gamma-1}}\sqrt{\frac{2}{\gamma+1}\big(\rho_{\infty}^{\gamma-1}-\max u_{0,L}^2\big)+u_{0,L}^2(s)}}\,{\rm d}s\nonumber\\
&=\int_{0}^{L}\frac{s}{\big(\frac{\gamma-1}{\gamma+1}\frac{\max u_{0,L}^2(s)}{\rho_{\infty}^{\gamma-1}}+\frac{2}{\gamma+1}\big)^{\frac{1}{\gamma-1}}\sqrt{\frac{2}{\gamma+1}\big(\frac{\rho_{\infty}^{\gamma-1}}{\max u_{0,L}^2}-1\big)\frac{\max u_{0,L}^2}{u_{0,L}^2(s)}+1}}\,{\rm d}s.
\end{align}
When $\rho_{\infty}>\rho_{\infty}^{*}=\big(\max u_{0,L}^2\big)^{\frac{1}{\gamma-1}}$, one has
\begin{align}\label{3.11}
G(\underline{\varrho}_{1,L})\leq \int_{0}^{L}\frac{s}{\Big(\frac{\gamma-1}{\gamma+1}\frac{\max u_{0,L}^2(s)}{\rho_{\infty}^{\gamma-1}}+\frac{2}{\gamma+1}\Big)^{\frac{1}{\gamma-1}}\sqrt{\frac{2}{\gamma+1}\Big(\frac{\rho_{\infty}^{\gamma-1}}{\max u_{0,L}^2}-1\Big)+1}}\,{\rm d}s.
\end{align}

Define
\begin{align*}
\mathcal{G}(s)=\Big(\frac{\gamma-1}{\gamma+1}s+\frac{2}{\gamma+1}\Big)^{\frac{1}{\gamma-1}}\sqrt{\frac{2}{\gamma+1}\big(\frac{1}{s}-1\big)+1}.
\end{align*}
Direct calculations give
\begin{align*}
\mathcal{G}'(s)=\Big(\frac{\gamma-1}{\gamma+1}s+\frac{2}{\gamma+1}\Big)^{\frac{1}{\gamma-1}}\sqrt{\frac{2}{\gamma+1}\Big(\frac{1}{s}-1\Big)+1}\frac{s-1}{s(2+s(\gamma-1))},
\end{align*}
so $\mathcal{G}(s)$ is a decreasing function $[0,1]$. Noting $\mathcal{G}(1)=1$, one has that for any $\varepsilon>0$, there exists an $\bar{L}>0$ such that if $L>\bar{L}$, it holds that
$$
\mathcal{G}\Big(\frac{1}{(1+\varepsilon)^{\gamma-1}}\Big)\geq \frac{L^2}{L^2-J^2}.
$$
Therefore, for any $\rho_{\infty}>(1+\varepsilon)\rho_{\infty}^{*}$, there exists a unique $\rho\in (\underline{\varrho}_{1,L},\bar{\varrho}_{1,L})$ satisfying $G(\rho)=\frac{1}{2}(L^2-J^2)$.

Define $\xi(s)=\chi^2(s)-s^2$. It follows from \eqref{3.7} that
\begin{align}\label{3.12}
\xi'(s)=\frac{2(\rho_{\infty}u_{0,L}(s)-\rho_{1,L}u_{1,L}(s))s}{\rho_{1,L}u_{1,L}(\chi(s))}.	
\end{align}
Noting $\rho_{1,L}<\rho_{\infty}$, we have $\rho_{\infty}u_{0,L}(s)<\rho_{1,L}u_{1,L}(\chi(s))$ as the momentum is decreasing with the respect to $\rho$ in the subsonic region.
Therefore, $\xi(s)<0$ for $s>0$, which yields that
$0=\xi(L)\leq \xi(s)=\chi(s)^2-s^2\leq \xi(0)=J^2$. That is, for any $r\in [J,L]$, it holds that
$$
0\leq r^2-(\chi^{-1}(r))^2\leq J^2.
$$
Then, for $r\in [J,L]$, it follows from a straightforward calculation that
\begin{align*}
\bar{\psi}(r)-\hat{\psi}(r)&=\int_{0}^{r}\rho_{\infty}u_{0,L}(s)s\,{\rm d}s-\int_{J}^{r}\rho_{1,L}u_{1,L}(s)s\,{\rm d}s\\
&=\int_{0}^{r}\rho_{\infty}u_{0,L}s\,{\rm d}s-\int_{\chi^{-1}(J)}^{\chi^{-1}(r)}\rho_{1,L}u_{1,L}(\chi(s))\chi'(s)\,{\rm d}s\\
&=\int_{0}^{r}\rho_{\infty}u_{0,L}(s)s\,{\rm d}s-\int_{0}^{\chi^{-1}(r)}\rho_{0,L}u_{0,L}(s)s\,{\rm d}s\\
&=\int_{\chi^{-1}(r)}^{r}\rho_{\infty}u_{0,L}(s)s\,{\rm d}s,
\end{align*}
where \eqref{3.7} has been used in the third equality. Since $0\leq r^2-(\chi^{-1}(r))^2\leq J^2$, one gets
$$
0\leq \bar{\psi}(r)-\hat{\psi}(r)\leq C\rho_{\infty},
$$
where the constant $C$ depends only on $J$ and $\max u_{\infty}$. This completes the proof of Proposition \ref{prop1}. $\hfill\square$

\begin{proposition}\label{prop2}
Let $\hat{\Omega}=\{(x,r)\,|\,r\in (J,L),x\in \R\}$. For any $\varepsilon>0$, there exists an $\bar{L}>0$ such that if $L>\bar{L}$, then for any $\rho_{\infty}\in ((1+\varepsilon)\rho_{0}^{*},\infty)$, if $\psi_{L}$ is a subsonic solution of the problem \eqref{2.9}, then it holds that
\begin{align}\label{4.1}
	\psi_{L}\geq \hat{\psi}_{L}\text{ in }\hat{\Omega}_{L}\text{ and }0\leq \psi_{L}\leq \bar{\psi}_{L}\text{ in }\Omega_{L},
\end{align}
where $\hat{\psi}_{L}$ and $\bar{\psi}_{L}$ are defined in \eqref{3.4}.
\end{proposition}

\noindent\textbf{Proof}. Define
\begin{align}\label{4.2}
\Psi_{L}(x,r)=\psi_{L}(x,r)-\bar{\psi}_{L}(r)\,\,\text{for }x\in \Omega\text{ and }\hat{\Psi}_{L}(x,r)=\hat{\psi}_{L}(r)-\psi_{L}(x,r)\quad \text{for }x\in \hat{\Omega}_{L}.
\end{align}
Noting the definition of $\hat{\psi}$ in \eqref{3.4}, one has $\kappa_{L}(\hat{\psi};\rho_{1,L})=r$,
$H_{L}=H\Big(\frac{|\nabla \hat{\psi}|}{r},\hat{\psi};\rho_{1,L}\Big)=\rho_{1,L}$, and
\begin{align}\label{4.3}
\Big(\Big(1-\frac{|\nabla \hat{\psi}_{L}|^2}{r^2H_{L}^{\gamma+1}}\Big)\delta_{ij}+\frac{1}{H_{L}^{\gamma+1}}\frac{\hat{\psi}_{i}}{r}\frac{\hat{\psi}_{j}}{r}\Big)\partial_{ij}\hat{\psi}_{L}-\frac{\partial_{r}\hat{\psi}_{L}}{r}=r\rho_{1,L}u_{1,L}'(r),
\end{align}
that is, $\hat{\psi}_{L}$ satisfies \eqref{2.9} with boundary conditions $\hat{\psi}_{L}(r)=0$ at $r=J$ and $\hat{\psi}_{L}(r)=m_{L}$ at $r=L$. Then $\hat{\Psi}_{L}$ will satisfy
\begin{align}\label{4.4}
\left\{
\begin{aligned}
	&\partial_{i}\Big(a_{ij}(\nabla\psi_{L},\psi;\nabla\hat{\psi}_{L},\hat{\psi}_{L})\partial_{j}\hat{\Psi}_{L}+b_{i}(\nabla\psi_{L},\psi_{L};\nabla\hat{\psi}_{L},\hat{\psi}_{L})\hat{\Psi}_{L}\Big)\\
	&\quad =c_{i}(\nabla\psi_{L},\psi_{L};\nabla\hat{\psi}_{L},\hat{\psi}_{L})\partial_{i}\hat{\Psi}_{L}+d(\nabla\psi_{L},\psi_{L};\nabla\hat{\psi}_{L},\hat{\psi}_{L})\hat{\Psi}_{L},\\
	&\hat{\Psi}_{L}\leq 0\quad \text{at }r=J,\quad \text{and }\hat{\Psi}_{L}=0\quad \text{at }r=L,
\end{aligned}
\right.
\end{align}
where
\begin{align}\label{4.4-2}
a_{ij}(\nabla\psi_{L},\psi_{L};\nabla\hat{\psi}_{L},\hat{\psi}_{L})&=\int_{0}^{1}\frac{1}{rH_{L}^2(\mathcal{M}_{L,t},\hat{\psi}_{L,t})}\Big(H_{L}(\mathcal{M}_{L,t},\hat{\psi}_{L,t})\delta_{ij}-\frac{\partial H_{L}}{\partial \mathcal{M}}\frac{2\partial_{i}\hat{\psi}_{L,t}\partial_{j}\hat{\psi}_{L,t}}{r^2}\Big)\,{\rm d}t,\nonumber\\
b_{i}(\nabla{\psi}_{L},\psi_{L};\nabla \hat{\psi}_{L},\hat{\psi}_{L})&=-\int_{0}^{1}\frac{\partial H_{L}}{\partial \mathcal{\psi}}\frac{\partial_{i}\hat{\psi}_{L,t}}{rH_{L}^2(\mathcal{M}_{L,t},\hat{\psi}_{L,t})}\,{\rm d}t,\nonumber\\
c_{i}(\nabla\psi_{L},\psi_{L};\nabla\hat{\psi}_{L},\hat{\psi}_{L})&=2\int_{0}^{1}\theta_{L}(\hat{\psi}_{L,t})\theta_{L}'(\hat{\psi}_{L,t})\frac{\partial H_{L}}{\partial \mathcal{M}}\frac{\partial_{i}\hat{\psi}_{L,t}}{r}\,{\rm d}t,\nonumber\\
d(\nabla\psi_{L},\psi_{L};\nabla\hat{\psi}_{L},\hat{\psi}_{L})&=r\int_{0}^{1}\big(\theta_{L}(\hat{\psi}_{L,t})\theta_{L}''(\hat{\psi}_{L,t})+(\theta_{L}'(\hat{\psi}_{L,t}))^2\big)H_{L}(\mathcal{M}_{L,t},\hat{\psi}_{L,t})\,{\rm d}t\nonumber\\
&\quad +r\int_{0}^{1}\theta_{L}(\hat{\psi}_{L,t})\theta_{L}'(\hat{\psi}_{L,t})\frac{\partial H_{L}}{\partial \psi}\,{\rm d}t\nonumber\\
&=:d_{1}(\nabla\psi_{L},\psi_{L};\nabla\hat{\psi}_{L},\hat{\psi}_{L})+d_{2}(\nabla\psi_{L},\psi_{L};\nabla\hat{\psi}_{L},\hat{\psi}_{L}),
\end{align}
for $i,j=1,2$, with $\hat{\psi}_{L,t}=t\hat{\psi}_{L}+(1-t)\psi_{L}$ and $\mathcal{M}_{L,t}=\frac{|\nabla \hat{\psi}_{L,t}|^2}{r^2}$ for $t\in [0,1]$. Here and in what follows, we neglect the parameter $\rho_{\infty}$ in the coefficients, unless confusion arises.

Similar to \eqref{1.12-1}, a straightforward calculation shows that
\begin{align}\label{4.4-1}
\frac{\partial H_{L}}{\partial\mathcal{M}}=-\frac{H_{L}(\mathcal{M}_{L,t},\hat{\psi}_{L,t})}{2(H_{L}^{\gamma+1}(\mathcal{M}_{L,t},\hat{\psi}_{L,t})-\mathcal{M}_{L,t})},\quad \frac{\partial H_{L}}{\partial\psi}=\frac{\theta_{L}(\hat{\psi}_{L,t})\theta_{L}'(\hat{\psi}_{L,t})H_{L}^3(\mathcal{M}_{L,t},\hat{\psi}_{L,t})}{(H_{L}^{\gamma+1}(\mathcal{M}_{L,t},\hat{\psi}_{L,t})-\mathcal{M}_{L,t})},
\end{align}
which implies $b_{i}=c_{i}$.

Set $\hat{\Psi}_{L}^{+}=\max\{\hat{\Psi}_{L},0\}$. Multiplying \eqref{4.4} by $\hat{\Psi}_{L}^{+}$ and integrating by parts imply that
\begin{align}\label{4.5}
&\iint_{\hat{\Omega}_{L}}\Bigg[|\nabla\hat{\Psi}_{L}^{+}|^2\int_{0}^{1}\frac{1}{rH_{L}(\mathcal{M}_{L,t},\hat{\psi}_{L,t})}\,{\rm d}t\nonumber\\
&\quad +\int_{0}^{1}\frac{1}{r^3H_{L}(\mathcal{M}_{L,t},\hat{\psi}_{L,t})(H_{L}^{\gamma+1}(\mathcal{M}_{L,t},\hat{\psi}_{L,t}))}\Big[r^2H_{L}^2\theta_{L}\theta_{L}'\hat{\Psi}_{L}^{+}-\nabla\hat{\psi}_{L,t}\cdot \nabla\hat{\Psi}_{L}^{+}\Big]^2\,{\rm d}t\Bigg]\,{\rm d}x{\rm d}r\nonumber\\
&=-\iint_{\hat{\Omega}_{L}}r\Big(\int_{0}^{1}\big(\theta_{L}(\hat{\psi}_{L,t})\theta_{L}''(\hat{\psi}_{L,t})+(\theta_{L}'(\hat{\psi}_{L,t}))^2\big)H_{L}(\mathcal{M}_{L,t},\hat{\psi}_{L,t})\,{\rm d}t\Big)|\hat{\Psi}_{L}^{+}|^2\,{\rm d}x{\rm d}r.
\end{align}
A direct calculation shows that 
\begin{align*}
&\theta_{L}(\hat{\psi}_{L,t})\theta_{L}''(\hat{\psi}_{L,t})+(\theta_{L}'(\hat{\psi}_{L,t}))^2\\
&=\frac{1}{\kappa^3(\hat{\psi}_{L,t})\rho_{\infty}^2u_{\infty}(\kappa(\hat{\psi}_{L,t}))}\big[u_{\infty,L}''(\kappa(\hat{\psi}_{L,t}))\kappa(\hat{\psi}_{L,t})-u_{\infty,L}'(\kappa(\hat{\psi}_{L,t}))\big]\geq 0,
\end{align*}
where we have used 
the structural condition \eqref{1.3-4}. Hence, it follows from \eqref{4.5} that
$$
\nabla\hat{\Psi}_{L}^{+}\equiv 0\quad \text{in }\hat{\Omega}_{L}.
$$
Since $\hat{\Psi}_{L}^{+}=0$ on $\partial\hat{\Omega}_{L}$. Thus, $\hat{\Psi}_{L}^{+}\leq 0$ in $\hat{\Omega}_{L}$, which is equivalent to $\psi_{L}\geq \hat{\psi}_{L}$ in $\hat{\Omega}$. Similarly, one can show that $\psi_{L}\leq \bar{\psi}_{L}$ in $\Omega_{L}$. This completes the proof of Proposition \ref{prop2}. $\hfill\square$.

\begin{lemma}\label{lem1}
For any $\varepsilon>0$, there exists an $\bar{L}>0$ such that if $L>\bar{L}$, then for any $\rho_{\infty}\in ((1+\varepsilon)\rho_{\infty}^{*},\infty)$, if $\psi_{L}$ is a subsonic solution of the problem \eqref{2.9} satisfying $0\leq \psi_{L}\leq m_{L}$, then it holds that
\begin{align}\label{4.6}
	-C\rho_{\infty}\leq \Psi_{L}\leq 0,
\end{align}
where the constant $C$ depends only on $J$ and $\max u_{\infty}$.
\end{lemma}

\noindent\textbf{Proof}. It follows directly from \eqref{4.5} that $\Psi_{L}=\psi_{L}-\bar{\psi}_{L}\leq 0$. For $r\geq J$, we obtain from Propositions \ref{prop1} and \ref{prop2} that
\begin{align*}
\Psi_{L}=\psi_{L}-\bar{\psi}_{L}\geq \hat{\psi}_{L}-\bar{\psi}_{L}\geq -C\rho_{\infty}.
\end{align*}
And for $r\in (0,J)$, then it follows from $\psi_{L}\geq 0$ that
\begin{align*}
\Psi_{L}\geq \psi_{L}-\bar{\psi}_{L}\geq -\bar{\psi}_{L}\geq -C\rho_{\infty}.
\end{align*}
This concludes the proof of Lemma \ref{lem1}. $\hfill\square$

\begin{lemma}\label{lem2}
For any $\varepsilon>0$, there exists an $\bar{L}>0$ such that if $L>\bar{L}$, then for any $\rho_{\infty}\in ((1+\varepsilon)\rho_{\infty}^{*},\infty)$, if $\psi_{L}$ is a subsonic solution of the problem \eqref{2.9} satisfying $0\leq \psi_{L}\leq m_{L}$, then
\begin{align}\label{4.7-5}
	\|\Psi_{L}\|_{C^{1}(\bar{\Omega}_{L})}\leq \mathscr{C}\rho_{\infty}r,\quad \|\Psi_{L}\|_{C^{1,\beta}(\overline{\Omega}_{L})}\leq \mathscr{C}\rho_{\infty}\max\{r,r^{1-\beta}\},
\end{align}
where $\mathscr{C}$ depends only on $J$, $\max u_{\infty}$ and $\varepsilon_{0}$ but is independent of $L$.
\end{lemma}

\noindent\textbf{Proof}. Recall the definition of \eqref{3.4}. A direct calculation shows that $\Psi_{L}$ satisfies
\begin{align}\label{4.7}
\left\{
\begin{aligned}
	&\Big(\Big(1-\frac{|\nabla\psi_{L}|^2}{r^2H_{L}^{\gamma+1}}\Big)\delta_{ij}+\frac{\partial_{i}\psi_{L}\partial_{j}\psi_{L}}{r^2H_{L}^{\gamma+1}}\Big)\partial_{ij}\Psi_{L}\\
	&\qquad=r^2\theta_{L}\theta_{L}'H_{L}^2-\rho_{\infty}ru_{\infty}'(r)+\frac{|\partial_{1}\psi_{L}|^2}{r^2H_{L}^{\gamma+1}}\rho_{\infty}(u_{\infty}(r)+ru_{\infty}'(r))+\frac{\partial_{2}\Psi_{L}}{r},\\
	&\Psi_{L}=-\int_{0}^{f(x)}\rho_{\infty}u_{\infty,L}(s)s\,{\rm d}s\quad \text{on } \Gamma,\quad \Psi_{L}=0 \text{ on }\Gamma_{L}.
\end{aligned}
\right.
\end{align}
Noting that $0\leq \bar{\psi}_{L}-\psi_{L}\leq C\rho_{\infty}$ and $\kappa_{L}(\bar{\psi}_{L};\rho_{\infty})=r$, we have
$$
r^2-\kappa_{L}^2(\psi_{L};\rho_{\infty})=\int_{\psi_{L}}^{\bar{\psi}_{L}}2\kappa_{L}(s;\rho_{\infty})\kappa_{L}'(s;\rho_{\infty})\,{\rm d}s=\frac{2}{\rho_{\infty}}\int_{\psi_{L}}^{\bar{\psi}_{L}}\frac{1}{u_{\infty,L}(s)}\,{\rm d}s\leq C,
$$
which implies that
$$
\frac{r}{\kappa_{L}(\psi_{L};\rho_{\infty})}\leq C\Big(1+\frac{1}{\kappa_{L}(\psi_{L};\rho_{\infty})}\Big).
$$
Therefore
$$
r^2\theta_{L}\theta_{L}'H_{L}^2=r\frac{r}{\kappa_{L}(\psi_{L};\rho_{\infty})}\frac{u_{\infty,L}'(\kappa_{L}(\psi_{L};\rho_{\infty}))H_{L}^2(\mathcal{M},\psi_{L};\rho_{\infty})}{\rho_{\infty}}\leq C\rho_{\infty}r,
$$
where we have used the facts $H_{L}(\mathcal{M},\psi_{L};\rho_{\infty})\leq C\rho_{\infty}$ and $\frac{u_{\infty}'(r)}{r}$ is uniformly bounded.

\textbf{Step 1. Estimates away from the axis $\Gamma$}. For any $r>\delta>0$, noting that the source term in \eqref{4.7} is bounded by $C_{\delta}\rho_{\infty}r$, we obtain from the H\"{o}lder gradient estimate \cite[Theorem 12.4]{GT} for linear uniform elliptic equation and \eqref{4.6} that
\begin{align}\label{4.7-1}
\|\Psi\|_{C^{1,\beta}(\Omega)}\leq \mathscr{C}|\Psi|+C_{\delta}\rho_{\infty}r\leq \mathscr{C}_{\delta}\rho_{\infty}r.
\end{align} 

\textbf{Step 2. Estimates near the axis $\Gamma$}. For $r<\delta$ near the axis, one has
\begin{align}\label{4.7-2}
\psi_{L}(r)\leq \bar{\psi}_{L}(r)\leq C\rho_{\infty}r^2.
\end{align}
For any fixed point $(x_{0},r_{0})\in \R\times (0,\delta)$, set
\begin{align*}
\tilde{\psi}_{L}(x,r)=\frac{1}{\xi^2}\psi_{L}(x_{0}+x\xi,r_{0}+r\xi),\quad \xi=\frac{r_{0}}{2},
\end{align*}
which is well defined in $B_{1}(0,0)$. Moreover, direct calculations yield that
\begin{align*}
\frac{\nabla \psi_{L}(x_{0}+x\xi,r_{0}+r\xi)}{r_{0}+r\xi}=\frac{\nabla \tilde{\psi}_{L}}{2+r},
\end{align*}
and $\tilde{\psi}_{L}$ satisfies
\begin{align}\label{4.7-3}
\operatorname{div}\Big(\frac{\nabla\tilde{\psi}_{L}}{(2+r)H_{L}\big(\frac{|\nabla \tilde{\psi}_{L}|^2}{(2+r)^2},\xi^2\tilde{\psi}_{L};\rho_{\infty}\big)}\Big)=(r_{0}+r\xi)\theta\theta'H_{L}\big(\frac{|\nabla \tilde{\psi}_{L}|^2}{(2+r)^2},\xi^2\tilde{\psi}_{L};\rho_{\infty}\big).
\end{align}
Due to \eqref{4.7-2}, one gets
\begin{align*}
0\leq \tilde{\psi}_{L}(x,r)=\frac{4}{r_{0}^2}\psi_{L}(x_{0}+x\xi,r_{0}+r\xi)\leq C\rho_{\infty}.
\end{align*}
Applying Moser's iteration, we get
$$
|\nabla\tilde{\psi}_{L}|\leq \mathscr{C}\rho_{\infty},
$$
which yields that 
\begin{align}\label{4.7-4}
\frac{|\nabla \psi_{L}(x_{0},r_{0})|}{r_{0}}=\frac{1}{2}|\nabla \tilde{\psi}_{L}(0,0)|\leq \mathscr{C}\rho_{\infty}.
\end{align}
Combining \eqref{4.7-1} and \eqref{4.7-4}, we obtain \eqref{4.7-5}. This concludes the proof of Lemma \ref{lem2}. $\hfill\square$.

\medskip

Moreover, we have the following far-field behavior and uniform integral estimate for $\Psi_{L}$.

\begin{lemma}\label{lem3}
If $\psi_{L}$ is a subsonic solution of the problem \eqref{2.9}, then we have
\begin{align}\label{4.8}
	|\psi_{L}-\bar{\psi}_{L}|\to 0\quad \text{uniformly with respect to $r\in [0,L]$, as $|x|\to \infty$},
\end{align}
and
\begin{align}\label{4.9}
	\Big\|\frac{|\nabla(\psi_{L}-\bar{\psi}_{L})|}{r^{\frac{1}{2}}}\Big\|_{L^2(\Omega_{L})}\leq \mathscr{C}\rho_{\infty},
\end{align}
where the constant $\mathscr{C}$ depends only on $J$, $\max u_{\infty}$ and $\varepsilon_{0}$, but is independent of $L$. 
\end{lemma}

\noindent\textbf{Proof}. First, the asymptotic behavior \eqref{4.8} follows from \cite[Section 4]{DD-2011}. For \eqref{4.9}, similar to \eqref{4.4}, $\Psi_{L}=\psi_{L}-\bar{\psi}_{L}$ satisfies
\begin{align}\label{4.10}
\left\{
\begin{aligned}
	&\partial_{i}\Big(a_{ij}(\nabla\psi_{L},\psi;\nabla\Psi_{L},\Psi_{L})\partial_{j}\Psi_{L}+b_{i}(\nabla\psi_{L},\psi_{L};\nabla\Psi_{L},\Psi_{L})\Psi_{L}\Big)\\
	&\quad =c_{i}(\nabla\psi_{L},\psi_{L};\nabla\Psi_{L},\Psi_{L})\partial_{i}\Psi_{L}+d(\nabla\psi_{L},\psi_{L};\nabla\Psi_{L},\Psi_{L})\Psi_{L},\\
	&\Psi_{L}=-\bar{\psi}_{L}(r)\quad \text{on }\Gamma,\quad \text{and }\Psi_{L}=0\quad \text{on }\Gamma_{L},
\end{aligned}
\right.
\end{align}
where the coefficients $a_{ij}$, $b_{i}$, $c_{i}$ and $d_{i}$ are similarly defined in \eqref{4.4-2} with $\hat{\psi}_{L}$ replaced by $\bar{\psi}_{L}$.

For any large constant $K>0$, multiplying \eqref{4.10} by $\Psi_{L}$ and integrating over $\Omega_{L,K}=\Omega_{L}\cap \{|x|\leq K\}$ yield that
\begin{align}\label{4.12}
&	\iint_{\Omega_{L,K}}\big(a_{ij}\partial_{i}\Psi_{L}\partial_{j}\Psi_{L}+2b_{i}\Psi_{L}\partial_{i}\Psi_{L}+d_{1}\Psi_{L}\big)\,{\rm d}x{\rm d}r\nonumber\\
&=\int_{\partial\Omega_{L,K}}(a_{ij}\partial_{j}\Psi_{L}+b_{i}\Psi_{L})\Psi_{L}n_{i}\,{\rm d}S-\iint_{\Omega_{L,K}}d_{2}\Psi_{L}^2\,{\rm d}x{\rm d}r\nonumber\\
&=-\int_{r=f(x)}(a_{ij}\partial_{j}\Psi_{L}+b_{i}\Psi_{L})\Psi_{L}n_{i}\,{\rm d}S-\iint_{\Omega_{L,K}}d_{2}\Psi_{L}^2\,{\rm d}x{\rm d}r\nonumber\\
&\quad -\iint_{x=K}(a_{ij}\partial_{j}\Psi_{L}+b_{i}\Psi_{L})\Psi_{L}n_{i}\,{\rm d}S+\iint_{x=-K}(a_{ij}\partial_{j}\Psi_{L}+b_{i}\Psi_{L})\Psi_{L}n_{i}\,{\rm d}S,
\end{align}
where ${\bf n}=(n_{1},n_{2})$ is the unit normal vector at the boundary $\partial\Omega_{L,K}$. Noting that $d_{2}\geq 0$, we have
\begin{align}\label{4.13}
&	\iint_{\Omega_{L,K}}\big(a_{ij}\partial_{i}\Psi_{L}\partial_{j}\Psi_{L}+2b_{i}\Psi_{L}\partial_{i}\Psi_{L}+d_{1}\Psi_{L}\big)\,{\rm d}x{\rm d}r\nonumber\\
&\leq \mathscr{C}\Big(\int_{\R}|\bar{\psi}_{L}(f(x))|\,{\rm d}x+\int_{f(K)}^{L}|\Psi_{L}(K,r)|\,{\rm d}r+\int_{f(-K)}^{L}|\Psi_{L}(-K,r)|\,{\rm d}r\Big),
\end{align}
where we have used the uniform bounds for $\frac{|\nabla \Psi_{L}|}{r}$ and $\Psi_{L}$ established in Lemmas \ref{lem1}--\ref{lem2}.

Using similar calculation to \eqref{4.5}, we have
\begin{align}\label{4.14}
\iint_{\Omega_{L,K}}\frac{|\nabla \Psi_{L}|^2}{r}\,{\rm d}x{\rm d}r\leq \mathscr{C}\Big(\int_{\R}|\bar{\psi}_{L}(f(x))|\,{\rm d}x+\int_{0}^{L}|\Psi_{L}(K,r)|\,{\rm d}r+\int_{0}^{L}|\Psi_{L}(-K,r)|\,{\rm d}r\Big).
\end{align}
Due to the asymptotic behavior \eqref{4.8}, we have $\Psi_{L}(\pm K,r)\to 0$. Then, taking $K\to \infty$, we obtain from \eqref{4.14} that
\begin{align}\label{4.15}
\iint_{\Omega_{L}}\frac{|\nabla\Psi_{L}|^2}{r}\,{\rm d}x{\rm d}r\leq \mathscr{C}\int_{\R}|\bar{\psi}_{L}(f(x))|\,{\rm d}x\leq \mathscr{C}\rho_{\infty}.
\end{align}
This completes the proof of the Lemma \ref{lem3}. $\hfill\square$

\medskip

Similar to the Step 6 in the proof of Lemma \ref{2.1}, we are going to show the constant $\rho_{\infty,L}^{*}$ in Lemma \ref{lem2.2} can be chosen independently of $L$. Given $\rho_{\infty}\in (\rho_{\infty}^{*},\infty)$, let $\mathcal{S}_{L}(\rho_{\infty})$ be the set of all solutions of the problem \eqref{2.9} associated with $\rho_{\infty}$. Define
\begin{align}\label{4.16}
\mathcal{Q}_{L}(\rho_{\infty})=\sup_{\psi_{L}\in \mathcal{S}_{L}(\rho_{\infty})}\sup_{(x,r)\in \overline{\Omega}_{L}}\frac{|\nabla \psi_{L}(x,r)|}{rH_{L}^{\frac{\gamma+1}{2}}\Big(\frac{|\nabla \psi_{L}|^2}{r^2},\psi_{L};\rho_{\infty}\Big)}.
\end{align}
Set
\begin{align}\label{4.17}
\bar{\rho}_{\infty,L}=\operatorname{inf}\Big\{s\,|\,\mathcal{Q}_{L}(\rho_{\infty})<1-2\varepsilon_{0}\text{ for any $\rho_{\infty}>s$}\Big\}.
\end{align}

\begin{lemma}\label{lem4}
If $\bar{\rho}_{\infty,L}>2\rho_{\infty}^{*}$, then $\mathcal{Q}_{L}(\bar{\rho}_{\infty,L})=1-2\varepsilon_{0}$.
\end{lemma}

\noindent\textbf{Proof}. First, it follows from the continuous dependence on the parameter $\rho_{\infty}$ for solution of uniformly elliptic equations that $\mathcal{Q}_{L}(\bar{\rho}_{\infty,L})\leq 1-2\varepsilon_{0}$. If $\mathcal{Q}_{L}(\bar{\rho}_{\infty,L})<1-2\varepsilon_{0}$ and $\rho_{\infty}>2\rho_{\infty}^{*}$, noting from Lemma \ref{lem2} that $\mathcal{Q}_{L}(\bar{\rho}_{\infty,L})\leq C\rho_{\infty}^{\frac{1-\gamma}{2}}$, there exists a $\delta>0$ such that $\mathcal{Q}_{L}(s)\leq 1-2\varepsilon_{0}$ for any $s\in (\bar{\rho}_{\infty,L}-\delta,\bar{\rho}_{\infty,L})$, which contradicts with the definition $\bar{\rho}_{\infty,L}$. Therefore, the proof of Lemma \ref{lem4} is complete. $\hfill\square$

\begin{lemma}\label{lem5}
There exists a $\bar{\rho}_{\infty}\in (\rho_{\infty}^{*},\infty)$ independent of $L$ such that if $\rho_{\infty}>\bar{\rho}_{\infty}$, there exists a subsonic solution $\psi_{L}$ of \eqref{2.9} satisfying
\begin{align}\label{4.18}
	\sup_{(x,r)\in \bar{\Omega}_{L}}\frac{|\nabla \psi_{L}|}{rH_{L}^{\frac{\gamma+1}{2}}\Big(\frac{|\nabla \psi_{L}|^2}{r^2},\psi_{L};\rho_{\infty}\Big)}< 1-2\varepsilon_{0}.
\end{align}
\end{lemma}

\noindent\textbf{Proof}. If $\bar{\rho}_{\infty,L}>2\rho_{\infty}^{*}$, it follows from Lemma \ref{lem4} that the problem \eqref{2.9} associated with $\rho_{\infty}=\bar{\rho}_{\infty,L}$ admits a solution $\psi_{L}$ satisfying
\begin{align}\label{4.19}
\sup_{(x,r)\in \bar{\Omega}_{L}}\frac{|\nabla\psi_{L}(x,r)|}{rH_{L}^{\frac{\gamma+1}{2}}(\frac{|\nabla \psi_{L}|^2}{r^2},\psi_{L};\rho_{\infty})}=1-2\varepsilon_{0}.
\end{align}
Then it follows from Lemma \ref{lem2} that
\begin{align*}
1-2\varepsilon_{0}= \frac{\sup_{(x,r)\in\Omega_{L}}\frac{|\nabla\psi_{L}|}{r}}{\inf_{(x,r)\in\Omega_{L}}H_{L}^{\frac{\gamma+1}{2}}\Big(\frac{|\nabla \psi_{L}|^2}{r^2},\psi_{L};\rho_{\infty}\Big)}\leq \frac{\mathscr{C}\rho_{\infty}}{C\rho_{\infty}^{\frac{\gamma+1}{2}}}=\mathscr{C}_{1}\rho_{\infty}^{\frac{1-\gamma}{2}},
\end{align*} 
where $\mathscr{C}_{1}$ is a constant independent of $L$. Therefore, we have
$$
\bar{\rho}_{\infty,L}\leq \big(\frac{\mathscr{C}_{1}}{1-2\varepsilon_{0}}\big)^{\frac{\gamma-1}{2}}.
$$
Choose
$$
\bar{\rho}_{\infty}=\max\big\{2\rho_{\infty}^{*},\big(\frac{\mathscr{C}_{1}}{1-2\varepsilon_{0}}\big)^{\frac{\gamma-1}{2}}\big\},
$$
it is easy to see that $\bar{\rho}_{\infty}$ is independent of $L$ and that if $\rho_{\infty}>\bar{\rho}_{\infty}$, the problem \eqref{2.9} admits a solution $\psi_{L}$ satisfying \eqref{4.18}. Therefore, the proof of Lemma \ref{lem5} is complete. $\hfill\square$

\medskip

\noindent\textbf{Proof of Proposition \ref{prop4.1}}: In view of Lemmas \ref{lem1}--\ref{lem5}, if $\rho_{\infty}>\bar{\rho}_{\infty}$,  the problem \eqref{2.9} admits a solution $\psi_{L}$ satisfying
\begin{align*}
&0\leq \psi_{L}\leq \bar{\psi}_{L},\quad |\psi_{L}-\bar{\psi}_{L}|\leq C\rho_{\infty},\quad \frac{|\nabla (\psi_{L}-\bar{\psi}_{L})|}{r}\leq \mathscr{C}\rho_{\infty},\\
&\Big\|\frac{\nabla (\psi_{L}-\bar{\psi}_{L})}{r^{\frac{1}{2}}}\Big\|_{L^2(\Omega_{L})}\leq \mathscr{C}\rho_{\infty},\quad \sup_{(x,r)\in \overline{\Omega}_{L}}\frac{|\nabla\psi_{L}|}{rH_{L}^{\frac{\gamma+1}{2}}\Big(\frac{|\nabla \psi_{L}|^2}{r^2},\psi_{L};\rho_{\infty}\Big)}< 1-2\varepsilon_{0},
\end{align*}
where $\mathscr{C}$ depends on $\varepsilon_{0}$ but is independent of $L$. By uniform estimates of $\psi_L$ in Lemmas 4.4–4.6,  taking $L\to \infty$, there exists a sub-sequence of $\{\psi_{L}\}$ still labeled by $\psi_{L}$ converging to $\psi$ satisfying
\begin{align*}
&0\leq \psi\leq \bar{\psi},\quad |\psi-\bar{\psi}|\leq C\rho_{\infty},\quad \frac{|\nabla (\psi-\bar{\psi})|}{r}\leq \mathscr{C}\rho_{\infty},\\
&\Big\|\frac{\nabla (\psi-\bar{\psi})}{r^{\frac{1}{2}}}\Big\|_{L^2(\Omega)}\leq \mathscr{C}\rho_{\infty},\quad \sup_{(x,r)\in \bar{\Omega}}\frac{|\nabla\psi|}{rH^{\frac{\gamma+1}{2}}(\mathcal{M},\psi;\rho_{\infty})}< 1-2\varepsilon_{0}.	    
\end{align*}
Here $\bar{\psi}(r)=\rho_{\infty}\int_{0}^{r}u_{\infty}(s)s\,{\rm d}s$, and $\psi$ is the solution to \eqref{1.13}. Therefore, the proof of Proposition \ref{prop4.1} is complete. $\hfill\square$

\section{Fine Properties of Subsonic Solutions}

In this section, we study properties of subsonic solutions constructed in Proposition \ref{prop4.1}, such as the asymptotic behaviors in far-fields, positivity of axial velocity, uniqueness of subsonic Euler flow and existence of critical density $\rho_{cr}$. Consequently, the proof of Theorem \ref{thm1} will be given at the end of the section.

\subsection{Asymptotic behavior at far fields} Noting that
$\Psi(x,r)=\psi(x,r)-\bar{\psi}(r)$ satisfies
\begin{align*}
&\Big(\Big(1-\frac{|\nabla\psi|^2}{r^2H^{\gamma+1}}\Big)\delta_{ij}+\frac{\partial_{i}\psi\partial_{j}\psi}{r^2H^{\gamma+1}}\Big)\partial_{ij}\Psi\\
&\quad =r^2\theta_{L}\theta_{L}'H_{L}^2-r\rho_{\infty}u_{\infty}+\frac{|\partial_{1}\psi|^2}{r^2H^{\gamma+1}}\rho_{\infty}(u_{\infty}(r)+ru_{\infty}'(r))+\frac{\partial_{2}\Psi}{r}.
\end{align*}
Since $\|\Psi\|_{L^{\infty}(\bar{\Omega})}\leq C\rho_{\infty}$, similar to the proof of Lemma \ref{lem2}, we have
\begin{align}\label{5.3}
\|\Psi\|_{C^{1}(\bar{\Omega})}\leq \mathscr{C}\rho_{\infty}r.
\end{align}
Therefore, if $r=0$, we directly have
$$
|\nabla \Psi(x,0)|=0\Rightarrow \lim\limits_{|x|\to \infty}|\nabla \Psi(x,0)|=0.
$$
Moreover, we claim that
\begin{align}\label{5.1}
\lim\limits_{\sqrt{x^2+r^2}\to \infty}\frac{|\nabla\Psi(x,r)|}{r^{\frac{1}{2}}}=0.
\end{align}

\small

We shall prove \eqref{5.1} by a contradiction argument. If \eqref{5.1} does not hold, then there exist a constant $\delta_{0}>0$ and a sequence $\{(x^{i},r^{i})\}_{i=1}^{\infty}$ with $\sqrt{|x^{i}|^2+|r^{i}|^2}\to 0$ as $i\to \infty$, such that $\frac{|\nabla\Psi(x_{0},r^{i})|}{(r^{i})^{\frac{1}{2}}}\geq \delta_{0}$. It follows from \eqref{5.3} that $\frac{|\nabla \Psi(x,r)|}{r^{\frac{1}{2}}}$ is uniformly continuous in $(x,r)$. Then there exists a uniform constant $\tau_{0}>0$ such that
$$
\frac{|\nabla\Psi(x_{0},r)|^2}{r}\geq \frac{\delta_{0}^2}{4}\quad \text{for any $(x,r)\in B_{\tau_{0}}(x^{i},r^{i})$}
$$
and
$$
B_{\tau_{0}}(x^{i},r^{i})\cap B_{\tau_{0}}=\emptyset\quad\text{for $i\neq j$}.
$$
Thus, it holds that
\begin{align*}
\iint_{\cup_{i}B_{\tau_{0}}(x^{i},r^{i})}\frac{|\nabla\Psi|^2}{r}\,{\rm d}x{\rm d}r=\sum\limits_{i=1}^{\infty}\iint_{B_{\tau_{0}(x^{i},r^{i})}}\frac{|\nabla \Psi|^2}{r}\,{\rm d}x{\rm d}r=\infty,
\end{align*}
which is a contradiction with \eqref{2.33} in Proposition \ref{prop4.1}.

In particular, it follows directly from \eqref{5.1} that for any $r>0$,
\begin{align}\label{5.1-1}
\lim\limits_{|x|\to \infty}\frac{|\nabla \Psi|}{r}=0\Rightarrow \lim\limits_{|x|\to \infty}|(\rho u-\rho_{\infty}u_{\infty}),\rho v|=0,
\end{align}
and
for any $x\in \R$,
\begin{align}\label{5.1-2}
\lim\limits_{r\to \infty}\frac{|\nabla \Psi|}{r}=0\Rightarrow \lim\limits_{r\to \infty}|(\rho u-\rho_{\infty}u_{\infty}),\rho v|=0.
\end{align}

\subsection{Positivity of axial velocity away from axis}

In view of Proposition \ref{Equivlence}, the flow $(\rho,u,v)=(H,\frac{\partial_{r}\psi}{rH},-\frac{\partial_{x}\psi}{rH})$ is indeed a solution of steady Euler equations \eqref{1.3} with \eqref{1.3-1}--\eqref{1.3-2}, if it satisfies \eqref{E1}--\eqref{E2}. By using the far-field behaviors \eqref{5.1-1}--\eqref{5.1-2} and similar arguments as in \cite[Section 4.2]{CDXX-2016}, \eqref{E1}--\eqref{E2} hold except for the positivity of axial velocity away from the axis. As mentioned previously, we are not able to derive the $U=:\partial_{r}\psi_{r}>0$ by applying $\partial_{r}$ to \eqref{1.13} and using the energy estimates, since the coefficients in \eqref{1.13} depend on $r$ and the associated term after applying $\partial_{r}$ is out of control in the energy estimates.

Therefore, to show $U=:\partial_{r}\psi>0$ away from the axis, we should combine the far-field behaviors \eqref{5.1-1}--\eqref{5.1-2} and the intrinsic strong maximum principle in the equation of $U$. In fact, it follows directly from \eqref{5.1-1}--\eqref{5.1-2} that for any $r>0$,
\begin{align}\label{6.2}
U=\partial_{r}\psi=r\rho u>0\quad \text{for $|x|\geq K$ with $K$ sufficiently large},
\end{align}
and for any $x\in \R$,
\begin{align}\label{6.3}
U=\partial_{r}\psi=r\rho u>0\quad \text{for $r\geq L$ with $L$ sufficiently large}.
\end{align}
On the other hand, noting \eqref{2.14-2}, we have
\begin{align}\label{6.4}
\partial_{r}\psi=\lim\limits_{L\to \infty}\partial_{r}\psi_{L}\geq 0\quad \text{for $(x,r)\in \overline{\Omega}\cap \{(x,r)\,|\,|x|\leq K,\,\,0<r\leq L\}$}.
\end{align}
Using similar calculations as in \eqref{2.43}, $U=\partial_{r}\psi$ satisfies
\begin{align}\label{6.5}
a_{ij}\partial_{ij}U+b_{i}\partial_{i}U+cU
=2r\Theta\Theta'H_{L}^2\leq 0.
\end{align}
Then, using \eqref{6.2}--\eqref{6.4} and applying the strong maximum principle without requiring the sign of $c$ to \eqref{6.5}, we can see that $U$ cannot attain the minimum value zero in $\Omega$, \textit{i.e.},
\begin{align}\label{6.6}
U=\partial_{r}\psi=r\rho u>0\quad \text{in }\Omega\Rightarrow u>0\quad \text{in }\Omega.
\end{align}
Moreover, since $\psi$ attains the minimum value on $\{(x,r)\,|\,r=f(x),\,\,x\in (0,1)\}$, we obtain from the Hopf point lemma that for $x\in (0,1)$,
$$
\partial_{r}\psi\vert_{\{(x,r)\,|r=f(x),x\in (0,1)\}}=(r\rho u)\vert_{\{(x,r)\,|r=f(x),x\in (0,1)\}}>0\Rightarrow u\vert_{\{(x,r)\,|r=f(x),x\in (0,1)\}}>0,
$$
which, together with \eqref{6.6}, implies that the axial velocity $u$ is positive away from the axis.

\subsection{The uniqueness of subsonic Euler flows past an axisymmetric obstacle}

In this subsection, we show the uniqueness of subsonic Euler solutions satisfying the boundary condition \eqref{1.3-1}--\eqref{1.3-2} and properties \eqref{1.3-5}--\eqref{1.3-9}.

Assume that $\psi^{(i)}\in C^{2,\beta}(\Omega)\cap C^{1,\beta}(\overline{\Omega})$ ($i=1,2$) solve the problem \eqref{1.13} and satisfy
\begin{align}\label{7.1}
&\psi^{(i)}-\bar{\psi}\in L^{\infty}(\Omega),\quad \Big\|\frac{\nabla (\psi^{(i)}-\bar{\psi})}{r^{\frac{1}{2}}}\Big\|_{L^2(\Omega)}\leq \mathscr{C}\rho_{\infty},\quad \text{and}\nonumber\\
\,\,&\frac{|\nabla \psi^{(i)}|^2}{r^2}< (1-2\varepsilon_{0})H^{\gamma+1}(\mathcal{M}_{i},\psi^{(i)};\rho_{\infty}),\quad \lim\limits_{|x|\to \infty}|\psi^{(i)}-\bar{\psi}|_{C^{0}(\Omega\cap \{(x,r)\,|\,r\leq L\})}=0,
\end{align}
for any fixed $L>0$. Set $\phi=\psi^{(1)}-\psi^{(2)}$. Then it holds that
\begin{align}\label{7.2}
\|\phi\|_{L^{\infty}(\Omega)}\leq C\rho_{\infty},\quad \Big\|\frac{\nabla \phi}{r^{\frac{1}{2}}}\Big\|_{L^{2}(\Omega)}\leq \mathscr{C}\rho_{\infty},\quad \lim\limits_{|x|\to \infty}|\phi|_{C^{0}(\Omega\cap \{(x,r)\,|\,r\leq L\})}=0.
\end{align}

Similar to \eqref{4.4}, $\phi$ solves
\begin{align}\label{7.3}
\left\{\begin{aligned}
	&\pa_{i}\big(a_{ij}(\nabla \psi^{(1)},\psi^{(1)};\nabla \psi^{(2)},\psi^{(2)})\partial_{j}\phi)+\partial_{i}\big(b_{i}(\nabla \psi^{(1)},\psi^{(1)};\nabla \psi^{(2)},\psi^{(2)})\phi\big)\\
	&\quad =c_{i}(\nabla \psi^{(1)},\psi^{(1)};\nabla \psi^{(2)},\psi^{(2)})\partial_{i}\phi+d(\nabla \psi^{(1)},\psi^{(1)};\nabla \psi^{(2)},\psi^{(2)})\phi\quad \text{in }\Omega,\\
	&\phi=0\quad \text{on }r=f(x),
\end{aligned}
\right.
\end{align}
where $a_{ij}$, $b_{i}$, $c_{i}$ and $d$ are similarly defined in \eqref{4.4-2} with $H_{L}$ and $\theta_{L}$ replaced by $H$ and $\Theta$ respectively. For simplicity of notation, we denote $H_{(t)}=H(\frac{|\nabla \psi_{(t)}|^2}{r^2},\psi_{(t)};\rho_{\infty})$ with $\psi_{(t)}=(1-t)\psi^{(1)}+t\psi^{(2)}$ for $t\in [0,1]$.

For any $R>2\max\{J,1\}$, let $\eta(x,r)$ be a smooth spherically symmetric function defined by
\begin{align}\label{7.5}
\eta(x,r)=\eta(\tilde{r})=\left\{
\begin{aligned}
	&1\qquad \qquad\qquad \text{for }(x,r)\in \{(x,r)\,|\,\tilde{r}\leq R\},\\
	&\frac{2\ln R-\ln \tilde{r}}{\ln R}\quad\,\text{for }(x,r)\in \{(x,r)\,|\,R\leq \tilde{r}\leq R^2\},\\
	&0,\qquad\qquad\quad\,\,\,\, \text{for }(x,r)\{(x,r)\,|\,\tilde{r}\geq R^2\},
\end{aligned}
\right.
\end{align}
where $\tilde{r}=:\sqrt{r^2+x^2}$. It is clearly that $|\nabla\eta|=\frac{1}{\tilde{r}\ln R}$. Denote $\Omega^{R}=\Omega\cap \{(x,r)\,|\,\tilde{r}\leq R\}$. Multiplying the both sides of the equation \eqref{7.3} by $\eta^2(x,r)\phi$ and integrating in $\Omega$ gives that
\begin{align}\label{7.6}
&\iint_{\Omega^{R^2}}\eta^2(x,r)\Bigg[\int_{0}^1\frac{|\nabla\phi|^2}{rH_{(t)}}+\frac{|\nabla\phi\cdot \nabla \psi_{(t)}-r^2H_{(t)}^2\Theta(\psi_{(t)})\Theta'(\psi_{(t)})\phi|^2}{r^3H_{(t)}(H_{(t)}-\mathcal{M}_{(t)})}\,{\rm d}t\Bigg]\,{\rm d}x{\rm d}r\nonumber\\
&\quad +\iint_{\Omega^{R^2}}\eta^{2}(x,r)\Big[\int_{0}^{1}r\phi^2\int_{0}^{1}H_{(t)}[\Theta\Theta''+\big(\Theta'\big)^2]\,{\rm d}t\Big]\,{\rm d}x{\rm d}r\nonumber\\
&=-2\iint_{\Omega^{R^2}}\Big(a_{ij}\eta\eta\pa_{i}\eta\phi\pa_{j}\phi-\eta\pa_{i}\eta\phi^2\int_{0}^{1}\frac{\pa_{i}\psi_{(t)}\Theta\Theta'H_{(t)}}{r(H_{(t)}-\mathcal{M}_{(t)})}\,{\rm d}t\Big)\,{\rm d}x{\rm d}r.
\end{align}
Note that the second term on the left hand side of \eqref{7.6} is non-negative due to the structural condition $u_{\infty}''(r)r\geq u_{\infty}'(r)$. Therefore, we obtain from \eqref{7.6} that
\begin{align}\label{7.7}
&\iint_{\Omega^{R^2}}\eta^2(x,r)\Bigg[\int_{0}^1\frac{|\nabla\phi|^2}{rH_{(t)}}+\frac{|\nabla\phi\cdot \nabla \psi_{(t)}-r^2H_{(t)}^2\Theta(\psi_{(t)})\Theta'(\psi_{(t)})\phi|^2}{r^3H_{(t)}(H_{(t)}-\mathcal{M}_{(t)})}\,{\rm d}t\Bigg]\,{\rm d}x{\rm d}r\nonumber\\
&\leq \mathscr{C}\iint_{\Omega^{R^2}}\Big[\eta|\phi|\frac{|\nabla \eta\cdot \nabla\phi|}{r}+\eta|\phi||\nabla \eta|\int_{0}^1|\Theta\Theta'H_{(t)}|\,{\rm d}t\Big]\,{\rm d}x{\rm d}r\nonumber\\
&=:I_{1}+I_{2},
\end{align}
where we have used the facts that $|a_{ij}|\leq \frac{1}{r}$, and $|\phi|+\frac{|\nabla \psi_{(t)}|}{r}\leq \mathscr{C}\rho_{\infty}$.

For $I_{1}$, we split it into two parts
\begin{align}\label{7.8}
I_{1}&=\mathscr{C}\iint_{\Omega^{R^2}\cap \{r\geq J\}}\eta|\phi|\frac{|\nabla \eta\cdot \phi|}{r}\,{\rm d}r{\rm d}x+\mathscr{C}\iint_{\Omega^{R^2}\cap \{r\leq J\}}\eta|\phi|\frac{|\nabla \eta\cdot \phi|}{r}\,{\rm d}r{\rm d}x=I_{1,1}+I_{1,2}.
\end{align}
For $I_{1,1}$, it follows from the H\"{o}lder inequality that
\begin{align}\label{7.8-1}
I_{1,1}\leq \mathscr{C}\Big(\iint_{\Omega^{R^2}\cap \{r\geq J\}}|\nabla \eta|^2\,{\rm d}r{\rm d}x\Big)^{\frac{1}{2}}\Big(\iint_{\Omega^{R^2}\cap \{r\geq J\}}\eta^2\frac{|\nabla\phi|^2}{r}\,{\rm d}x{\rm d}r\Big)^{\frac{1}{2}}\leq \mathscr{C}\|\nabla\eta\|_{L^2}\leq \frac{\mathscr{C}}{\ln R}.
\end{align}
For $I_{1,2}$, noting that $|x|\in [\frac{R}{2},R^2]$ if $(x,r)\in (\Omega^{R^2}\backslash \Omega^{R})\cap \{r\leq J\}$, one has
\begin{align}\label{7.8-2}
I_{1,2}&\leq \mathscr{C}\int_{R}^{R^2}\int_{f(x)}^{J}|\nabla \eta||\phi|\,{\rm d}r{\rm d}x\leq  \mathscr{C}\int_{R}^{R^2}\int_{f(x)}^{J}\frac{1}{x\ln R}|\phi|\,{\rm d}{r}{\rm d}x\nonumber\\
&\leq \mathscr{C}\|\phi{\bf 1}_{\{R\leq |x|\leq R^2\}}\|_{C^{0}(\Omega\cap \{(x,r)\,|\,r\leq J\})}.
\end{align}
Combining \eqref{7.8-1} and \eqref{7.8-2}, one has
\begin{align}\label{7.8-3}
I_{1}\leq \mathscr{C}\big(\frac{1}{\ln R}+\|\phi{\bf 1}_{\{R\leq |x|\leq R^2\}}\|_{C^{0}(\Omega\cap \{(x,r)\,|\,r\leq J\})}\big).
\end{align}
For $I_{2}$, we also split it into two parts
\begin{align}\label{7.9}
I_{2}&=\iint_{\Omega^{R^2}\cap \{r\geq J\}}\eta|\phi||\nabla \eta|\int_{0}^{1}|\Theta\Theta'H_{(t)}|\,{\rm d}t{\rm d}x{\rm d}r+\iint_{\Omega^{R^2}\cap \{r\leq J\}}\eta\phi|\nabla \eta|\int_{0}^{1}|\Theta\Theta' H_{(t)}|\,{\rm d}t{\rm d}r{\rm d}x\nonumber\\
&=I_{2,1}+I_{2,2}.
\end{align}
For $I_{2,2}$, using similar calculations as in \eqref{7.8-2}, one has
\begin{align}\label{7.9-2}
I_{2,2}\leq \mathscr{C}\|\phi{\bf 1}_{\{R\leq |x|\leq R^2\}}\|_{C^{0}(\Omega\cap \{(x,r)\,|\,r\leq J\})}.
\end{align}
For $I_{2,1}$, a direct calculation shows that
\begin{align}\label{7.9-1}
I_{2,1}&\leq \mathscr{C}\Big(\iint_{\Omega^{R^2}\cap \{r\geq J\}}|\nabla \eta|^2\,{\rm d}r{\rm d}x\Big)^{\frac{1}{2}}\Big(\iint_{\Omega^{R^2}\cap \{r\geq J\}}\eta^2\int_{0}^{1}\frac{r^4|\Theta\Theta'H_{(t)}|^2\phi^2H_{(t)}^{2}}{r^4H_{(t)}(H_{(t)}^{\gamma+1}-\mathcal{M}_{(t)})}\,{\rm d}t\,{\rm d}r{\rm d}x\Big)^{\frac{1}{2}}\nonumber\\
&\leq \mathscr{C}\big(\frac{1}{\ln R}\big)^{\frac{1}{2}}\Big(\iint_{\Omega^{R^2}\cap \{r\geq J\}}\eta^2\int_{0}^{1}\frac{|r^2H_{(t)}^2\Theta\Theta'\phi-\nabla\phi\cdot\nabla \psi_{(t)}|^2}{r^4H_{(t)}(H_{(t)}^{\gamma+1}-\mathcal{M}_{(t)})}\,{\rm d}t\,{\rm d}r{\rm d}x\Big)^{\frac{1}{2}}\nonumber\\
&\quad +\mathscr{C}\big(\frac{1}{\ln R}\big)^{\frac{1}{2}}\Big(\iint_{\Omega^{R^2}\cap \{r\geq J\}}\eta^2\int_{0}^{1}\frac{|\nabla\phi\cdot \nabla \psi_{(t)}|^2}{r^4H_{(t)}(H_{(t)}^{\gamma+1}-\mathcal{M}_{(t)})}\,{\rm d}t\,{\rm d}r{\rm d}x\Big)^{\frac{1}{2}}\nonumber\\
&\leq \mathscr{C}\big(\frac{1}{\ln R}\big)^{\frac{1}{2}}\big(I_{1}^{\frac{1}{2}}+I_{2,1}^{\frac{1}{2}}+I_{2,2}^{\frac{1}{2}}\big)+\mathscr{C}\big(\frac{1}{\ln R}\big)^{\frac{1}{2}},
\end{align}
where we have used \eqref{7.2}, \eqref{7.7} in the last inequality. Then we obtain from \eqref{7.8-3}, and \eqref{7.9-2}--\eqref{7.9-1} that
\begin{align}\label{7.9-3}
I_{2,1}\leq \mathscr{C}\big(\frac{1}{\ln R}+\|\phi{\bf 1}_{\{R\leq |x|\leq R^2\}}\|_{C^{0}(\Omega\cap \{(x,r)\,|\,r\leq J\})}\big).
\end{align}
Substituting \eqref{7.8-3}, \eqref{7.9-2} and \eqref{7.9-3} into \eqref{7.7}, we have
\begin{align}\label{7.9-4}
\iint_{\Omega^{R^2}}\eta^2\int_{0}^{1}\frac{|\nabla\phi|^2}{rH_{(t)}}\,{\rm d}t\leq \mathscr{C}\big(\frac{1}{\ln R}+\|\phi{\bf 1}_{\{R\leq |x|\leq R^2\}}\|_{C^{0}(\Omega\cap \{(x,r)\,|\,r\leq J)\}}\big).
\end{align}
Taking $R\to \infty$ and using \eqref{7.2}, we get from \eqref{7.9-4} that
\begin{align}\label{7.9-5}
|\nabla \phi|^2=0\quad \text{in $\Omega$}.
\end{align}
Since $\phi=0$ on $\Gamma$, one has $\phi\equiv 0$ in $\Omega$. This concludes the uniqueness of subsonic Euler flows satisfying the boundary condition \eqref{1.3-1}--\eqref{1.3-2} and properties \eqref{1.3-5}--\eqref{1.3-9}.

\subsection{Existence of the critical density in the upstream}

Now, we show that there exists a critical density $\rho_{{\rm cr}}$ such that there exists a subsonic solution as long as the density of upstream flows is greater than $\rho_{cr}$. The proof is very similar to \cite[Section 6]{CDXX-2016}, \cite[Section 6]{DD-2011} and \cite[Proposition 6]{XX-2010-2}. For self-containedness, we sketch the proof here.

\begin{proposition}\label{prop3}
There exists a critical value $\rho_{cr}>0$ such that if $\rho_{\infty}>\rho_{cr}$, there exists a unique $\psi$ which solves the following problem
\begin{align}\label{6.1-1}
	\left\{
	\begin{aligned}
		&\operatorname{div}\Big(\frac{\nabla \psi}{rH(\mathcal{M},\psi;\rho_{\infty})}\Big)=r\Theta(\psi;\rho_{\infty})\Theta'(\psi;\rho_{\infty})H(\mathcal{M},\psi;\rho_{\infty})\quad \text{in $\Omega$},\\
		&\psi=0\quad \text{on $\Gamma$},
	\end{aligned}
	\right.
\end{align}
and satisfies
\begin{align*}
	0\leq \psi\leq \bar{\psi},\quad |\psi-\bar{\psi}|\leq C\rho_{\infty}\quad \text{in }\bar{\Omega},\quad \Big\|\frac{|\nabla(\psi-\bar{\psi})|}{r^{1/2}}\Big\|_{L^2(\Omega)}\leq \mathscr{C}\rho_{\infty},
\end{align*}
and
$$
\mathcal{Q}(\rho_{\infty})=\sup_{(x,r)\in \bar{\Omega}}\frac{|\nabla \psi|}{rH^{\frac{\gamma+1}{2}}(\mathcal{M},\psi;\rho_{\infty})}<1.
$$
Moreover, either $\mathcal{Q}(\rho_{\infty})\to 1$ as $\rho_{\infty}\downarrow \rho_{cr}$ or there does not exist a $\sigma>0$ such that the problem \eqref{6.1-1} has a solution for all $\rho_{\infty}\in (\rho_{cr}-\sigma,\rho_{cr})$ and 
$$
\sup_{\rho_{\infty}\in (\rho_{cr}-\sigma,\rho_{cr})}\mathcal{Q}(\rho_{\infty})<1.
$$
\end{proposition}

\noindent\textbf{Proof}. Let $\{\varepsilon_{n}\}_{n=0}^{\infty}$ be a strictly decreasing positive sequence satisfying $\varepsilon_{0}\leq \frac{1}{4}$ and $\varepsilon_{n}\to 0$ as $n\to \infty$, and $\chi_{n}$ be a sequence of smooth increasing functions satisfying
\begin{align*}
\chi_{n}=\left\{
\begin{aligned}
	&z,\qquad\qquad\quad\text{if $|z|\leq 1-2\varepsilon_{n}$},\\
	&1-\frac{3}{2}\varepsilon_{n},\qquad\text{if }z\geq 1-\varepsilon_{n}.
\end{aligned}
\right.
\end{align*}
Given $\rho_{\infty}\in (\rho_{\infty}^{*},\infty)$ and $\varepsilon_{n}>0$, similar to Proposition \ref{prop4.1}, we can construct a solution $\psi$ to \eqref{6.1-1}. Moreover, there is a constant $\bar{\rho}_{\infty}^{n}$ depending on $\varepsilon_{n}$ such that $\rho_{\infty}>\bar{\rho}_{\infty}^{n}$, and it holds that
$$
\frac{|\nabla \psi|}{rH^{\frac{\gamma+1}{2}}(\mathcal{M},\psi;\rho_{\infty})}\leq 1-2\varepsilon_{n}.
$$
Let $\mathcal{S}^{n}(\rho_{\infty})$ be the set of all solutions of the problem \eqref{6.1-1}. Denote
\begin{align*}
\mathcal{Q}^{n}(\rho_{\infty})=\sup_{\psi\in \mathcal{S}^{n}(\rho_{\infty})}\sup_{(x,r)\in \Omega_{L}}\frac{|\nabla\psi|}{rH^{\frac{\gamma+1}{2}}(\mathcal{M},\psi;\rho_{\infty})},
\end{align*}
and
\begin{align*}
\rho_{\infty}^{n}=\inf\{s\,|\,\text{for any $\rho_{\infty}\geq s$, $\mathcal{Q}^{(n)}(\rho_{\infty})\leq 1-2\varepsilon_{n}$}\}.
\end{align*}
It is clear that $\rho_{\infty}^{n}\leq \bar{\rho}_{\infty}^{n}$. For any $\rho_{\infty}>\rho_{\infty}^{n}$, we have the following uniform estimates
\begin{align*}
0\leq \psi\leq \bar{\psi},\,\,|\psi-\bar{\psi}|\leq C\rho_{\infty}\quad \text{and}\quad  \Big\|\frac{|\nabla (\psi-\bar{\psi})|}{r^{\frac{1}{2}}}\Big\|_{L^2(\Omega)}\leq \mathscr{C}_{n}\rho_{\infty},
\end{align*}
where $\mathscr{C}_{n}$ depends on $\varepsilon_n$. It is easy to see that $\{\rho_{\infty}^{n}\}$ is a decreasing sequence. Define $\rho_{cr}=\operatorname{inf}{\rho}_{\infty}^{n}$. Using similar arguments as in \cite[Proposition 6]{XX-2010-2} and \cite[Section 6]{DD-2011}, we can prove that $\rho_{cr}$ is the desired critical value described in Proposition \ref{prop3}. Therefore, the proof of Proposition \ref{prop3} is complete. $\hfill\square$

\subsection{Proof of Theorem \ref{thm1}} As a direct consequence of Proposition \ref{prop4.1} and Subsections 5.1--5.4, The proof of Theorem \ref{thm1} is complete. $\hfill\square$

\section{Limit of Subsonic Flows: Proof of Theorem \ref{thm2}}

In this section, we give a proof of Theorem \ref{thm2}. Given a decreasing sequence of $\{\rho_{\infty}^{n}\}$ satisfying 
$$
\lim\limits_{n\to \infty}\rho_{\infty}^{n}=\rho_{cr},
$$
let $(\rho_{n},u_{n},v_{n})(x,r)\in (C^{1,\alpha}(\Omega)\cap C^{\alpha}(\bar{\Omega}))^3$ be the associated subsonic equation to \eqref{1.3} established in Theorem \ref{thm1} for $\rho_{\infty}=\rho_{\infty}^{n}$. Then
$$
(\rho_{n},\mathbf{U}_{n})(x,y,z)=\Big(\rho_{n}(x,r),u_{n}(x,r),v_{n}(x,r)\frac{y}{r},v_{n}(x,r)\frac{z}{r}\Big)
$$
is the steady solution of \eqref{1.1}. Moreover, by direct calculations, one has
\begin{itemize}
\item[(1)] the Mach numbers of the flows $\sqrt{u_{n}^2+v_{n}^2}/{\rho_{n}^{\frac{\gamma-1}{2}}}\leq 1$ almost everywhere in $\Omega$;
\item[(2)] the Bernoulli functions of the flows $\frac{u_{n}^2+v_{n}^2}{2}+h(\rho_{n})$ are uniformly bounded above and below;
\item[(3)] the vorticities of the flows $\operatorname{curl}\mathbf{U}_{n}$ are precompact in $W_{loc}^{-1,p}(\Omega)$ for $1<p\leq 2$.
\end{itemize}
Hence, using the compensated compactness framework \cite[Theorem 2.2]{CHW-2016} yields that there exists a subsequence still labeled by $(\rho_{n},\mathbf{U}_{n})$ converging $(\rho,\mathbf{U})$ almost everywhere in $\Omega$. Thus, $(\rho,\mathbf{U})$ also solves the Euler system \eqref{1.1} in the weak sense and the boundary condition \eqref{1.2+} in the sense of normal trace. This completes the proof of Theorem \ref{thm2}. $\hfill\square$


\section*{Acknowledgements}
Dehua Wang was supported in part by NSF grants DMS-2219384 and DMS-2510532.
Tian-Yi Wang's research was supported in part by the NSFC grant 12371223.



\end{document}